\documentclass[10pt,twocolumn,twoside]{IEEEtran}
\usepackage[utf8]{inputenc}
\usepackage{amsmath,amsthm,amsfonts,amssymb,bm,accents}
\usepackage{graphicx,cite,enumerate,afterpage,url}
\graphicspath{{Images/}}

\usepackage{color}
\definecolor{pennblue}{cmyk}{1,0.65,0,0.30}
\definecolor{pennred}{cmyk}{0,1,0.65,0.34}
\definecolor{mygreen}{rgb}{0.10,0.50,0.10}

\usepackage{blindtext}

\newcommand{\includesvg}[2][scale=1]{\includegraphics[#1]{#2.pdf}}

\usepackage[english]{babel}

\usepackage{algpseudocode,algorithm}
\algrenewcommand\algorithmicdo{}
\algrenewtext{EndFor}{\algorithmicend}
\algrenewtext{EndProcedure}{\algorithmicend}

\newtheorem{theorem}{Theorem}
\newtheorem*{theorem*}{Theorem}
\newtheorem{proposition}{Proposition}
\newtheorem{corollary}{Corollary}
\newtheorem{lemma}{Lemma}
\theoremstyle{definition}
\newtheorem{definition}{Definition}

\newtheorem*{problem*}{Problem}
\newtheorem{remark}{Remark}
\newtheorem*{remark*}{Remark}
\newtheoremstyle{assume}
  {3pt}% measure of space to leave above the theorem. e.g.: 3pt
  {3pt}% measure of space to leave below the theorem. e.g.: 3pt
  {}% name of font to use in the body of the theorem
  {}% measure of space to indent
  {\bf}% name of head font
  {}% punctuation between head and body
  { }% space after theorem head; " " = normal interword space
  {\thmname{#1}.\thmnumber{#2}\thmnote{ \textnormal{(\textit{#3})}}}% Manually specify head
\theoremstyle{assume}

\DeclareMathOperator{\E}{\mathbb{E}}

\DeclareMathOperator{\trace}{Tr}

\DeclareMathOperator*{\argmin}{argmin}

\DeclareMathOperator*{\minimize}{minimize}

\DeclareMathOperator{\lmin}{\lambda_\textup{min}}
\DeclareMathOperator{\lmax}{\lambda_\textup{max}}

\DeclareMathOperator{\blkdiag}{blkdiag}

\newcommand{\vect}[2]{\ensuremath{\left[\begin{array}{#1} #2 \end{array}\right]}}
\newcommand{\norm}[1]{\ensuremath{\left\| #1 \right\|}}
\newcommand{\abs}[1]{\ensuremath{{\left\vert #1 \right\vert}}}

\newcommand{\kron}{\ensuremath{\otimes}}

\def\nnil{\nil}
\newcounter{prob}
\newenvironment{prob}[1][\nil]{%
	\def\tmp{#1}
	\equation
	\ifx\tmp\nnil
		\refstepcounter{prob}
		\tag{P\Roman{prob}}
	\else
		\tag{\tmp}
	\fi
	\aligned%
}{%
	\endaligned\endequation%
}

\newcommand{\calA}{\ensuremath{\mathcal{A}}}
\newcommand{\calB}{\ensuremath{\mathcal{B}}}

\newcommand{\calD}{\ensuremath{\mathcal{D}}}
\newcommand{\calE}{\ensuremath{\mathcal{E}}}

\newcommand{\calG}{\ensuremath{\mathcal{G}}}

\newcommand{\calN}{\ensuremath{\mathcal{N}}}
\newcommand{\calO}{\ensuremath{\mathcal{O}}}

\newcommand{\calS}{\ensuremath{\mathcal{S}}}
\newcommand{\calT}{\ensuremath{\mathcal{T}}}

\newcommand{\calV}{\ensuremath{\mathcal{V}}}

\newcommand{\calX}{\ensuremath{\mathcal{X}}}
\newcommand{\calY}{\ensuremath{\mathcal{Y}}}

\newcommand{\bA}{\ensuremath{\bm{A}}}
\newcommand{\bB}{\ensuremath{\bm{B}}}
\newcommand{\bC}{\ensuremath{\bm{C}}}
\newcommand{\bD}{\ensuremath{\bm{D}}}
\newcommand{\bE}{\ensuremath{\bm{E}}}
\newcommand{\bF}{\ensuremath{\bm{F}}}

\newcommand{\bH}{\ensuremath{\bm{H}}}
\newcommand{\bI}{\ensuremath{\bm{I}}}

\newcommand{\bK}{\ensuremath{\bm{K}}}
\newcommand{\bL}{\ensuremath{\bm{L}}}
\newcommand{\bM}{\ensuremath{\bm{M}}}

\newcommand{\bO}{\ensuremath{\bm{O}}}
\newcommand{\bP}{\ensuremath{\bm{P}}}
\newcommand{\bQ}{\ensuremath{\bm{Q}}}
\newcommand{\bR}{\ensuremath{\bm{R}}}

\newcommand{\bV}{\ensuremath{\bm{V}}}

\newcommand{\bX}{\ensuremath{\bm{X}}}
\newcommand{\bY}{\ensuremath{\bm{Y}}}
\newcommand{\bZ}{\ensuremath{\bm{Z}}}

\newcommand{\bb}{\ensuremath{\bm{b}}}

\newcommand{\be}{\ensuremath{\bm{e}}}

\newcommand{\bq}{\ensuremath{\bm{q}}}

\newcommand{\bu}{\ensuremath{\bm{u}}}
\newcommand{\bv}{\ensuremath{\bm{v}}}
\newcommand{\bw}{\ensuremath{\bm{w}}}
\newcommand{\bx}{\ensuremath{\bm{x}}}
\newcommand{\by}{\ensuremath{\bm{y}}}
\newcommand{\bz}{\ensuremath{\bm{z}}}
\newcommand{\bPi}{\ensuremath{\bm{\Pi}}}
\newcommand{\bPhi}{\ensuremath{\bm{\Phi}}}

\newcommand{\setN}{\ensuremath{\mathbb{N}}}

\newcommand{\setR}{\ensuremath{\mathbb{R}}}
\newcommand{\setS}{\ensuremath{\mathbb{S}}}

\def\st/{\textsuperscript{st}}
\def\nd/{\textsuperscript{nd}}
\def\rd/{\textsuperscript{rd}}
\def\th/{\textsuperscript{th}}

\newcommand{\del}{\ensuremath{\partial}}

\newcommand{\bxh}{\ensuremath{\bm{{\hat{x}}}}}
\newcommand{\bxc}{\ensuremath{\bm{{\check{x}}}}}
\newcommand{\bxt}{\ensuremath{\bm{{\tilde{x}}}}}

\newcommand{\bxb}{\ensuremath{\bm{{\bar{x}}}}}
\newcommand{\byb}{\ensuremath{\bm{{\bar{y}}}}}

\newcommand{\bvb}{\ensuremath{\bm{{\bar{v}}}}}
\newcommand{\bzb}{\ensuremath{\bm{{\bar{z}}}}}

\newcommand{\bMt}{\ensuremath{\bm{{\tilde{M}}}}}

\newcommand{\bXd}{\ensuremath{\bm{{\dot{X}}}}}

\title{Approximate Supermodularity of\\Kalman Filter Sensor Selection}

\author{Luiz~F.~O.~Chamon, George J. Pappas, and Alejandro~Ribeiro%
\thanks{Department of Electrical and Systems Engineering, University of Pennsylvania.
 e-mail: \mbox{\texttt{luizf@seas.upenn.edu}}, \mbox{\texttt{\{luizf, pappasg, aribeiro\}@seas.upenn.edu}}.
This work was supported by ARL DCIST CRA W911NF-17-2-0181.
Part of the results in this paper appeared in~\cite{Chamon17m}.}%
}

\begin{document}
\maketitle
\begin{abstract}

This work considers the problem of selecting sensors in a large scale system to minimize the error in estimating its states. More specifically, the state estimation mean-square error~(MSE) and worst-case error for Kalman filtering and smoothing. Such selection problems are in general NP-hard, i.e., their solution can only be approximated in practice even for moderately large problems. Due to its low complexity and iterative nature, greedy algorithms are often used to obtain these approximations by selecting one sensor at a time choosing at each step the one that minimizes the estimation performance metric. When this metric is supermodular, this solution is guaranteed to be~$(1-1/e)$-optimal. This is however not the case for the MSE or the worst-case error. This issue is often circumvented by using supermodular surrogates, such as the~$\log\det$, despite the fact that minimizing the~$\log\det$ is not equivalent to minimizing the MSE. Here, this issue is addressed by leveraging the concept of approximate supermodularity to derive near-optimality certificates for greedily minimizing the estimation mean-square and worst-case error. In typical application scenarios, these certificates approach the $(1-1/e)$ guarantee obtained for supermodular functions, thus demonstrating that no change to the original problem is needed to obtain guaranteed good performance.

\end{abstract}

\section{INTRODUCTION}
\label{S:intro}

Often, the outputs of a dynamical system can only be partially observed due to a limited sensing budget and/or because the system scale makes instrumenting all outputs impractical. This issue is critical for distributed systems, where power and communication constraints further limit the number of sensors available~\cite{Mo11s, Rowaihy07s}, and is found in applications such as target tracking, field monitoring, power allocation, and biological systems analysis~\cite{Zhao02i, Zhang09d, Hero11s, Chen12o, Liu12o, Roy16s}. In these cases, we must choose which outputs to observe in order to best estimate the system internal states.

In general, solving this constrained output selection problem, often known as \emph{sensor selection}, is NP-hard~\cite{Krause08n, Das11s, Ranieri14n, Tzoumas16m, Olshevsky17o, Zhang17s, Ye18c}. Still, there exist instances for which its solution can be approximated. For instance, when the goal is to find the smallest set of outputs that makes the system observable, greedy sensor selection is near-optimal~\cite{Summers16s}. When only structural descriptions are available, the problem is also known to be NP-hard, although a variant in which the output sensing can be designed freely admits a polynomial time solution~\cite{Commault13i, Pequito15c, Pequito16f}. Nevertheless, observability alone may not guarantee good state estimates because it does not take into account how hard each system mode is to observe. To do so, we turn to measures of the state estimation error.

In this setting, natural choices of objectives, especially in Kalman filtering applications, are the state estimation mean-square error~(MSE) and the worst-case estimation error. However, approximation guarantees for these problems are scarce. In particular, typical guarantees for greedy search do not apply since these functions are not supermodular, a diminishing returns property used to obtain these results---see~\cite{Tzoumas16m, Olshevsky17o, Zhang17s, Ye18c} for counter-examples in the context of control theory. Moreover, though convex relaxations of these problems have been proposed, they too lack near-optimality certificates~\cite{Weimer08r, Joshi09s, Moshtagh09o, Wang13c, Dhingra14a, Liu16s, Zhang17s}. In fact, system-independent guarantees cannot be provided for optimizing the state estimation MSE unless~$\text{P} = \text{NP}$~\cite{Ye18c}.

A typical way to address this issue is to use surrogate supermodular objectives, such as the log-determinant of the error covariance matrix or some information-theoretic measure relating the selected output set and its complement~\cite{Krause08n, Moshtagh09o, Shamaiah10g, Ranieri14n, Tzoumas16s, Summers16s}. Though effective, there is no direct relation between the estimation MSE or the worst-case estimation error and these cost functions. As a matter of fact, despite its usefulness in minimizing the volume of the confidence ellipsoid, the log-determinant is a poor proxy for the MSE---see Remark~\ref{R:logdet}.

In this work, we take a page out of Tukey’s book and try to solve the ``right problem''~\cite{Box05s}, i.e., \emph{selecting sensors so as to directly minimize the estimation MSE or the worst-case estimation error}. To make this selection scalable and obtain performance guarantees, we prove that greedy sensor selection is near-optimal for minimizing the estimation MSE and the worst-case estimation error in both filtering and smoothing settings. Our results show that when the measurement noise is large compared to the process noise, a scenario of practical relevance in Kalman filtering, greedy solutions to these problems approach the classical~$(1-1/e)$-optimal certificate for supermodular objectives~\cite{Nemhauser78a}.

To derive these guarantees, we leverage two notions of \emph{approximate supermodularity} in which limited violations of the diminishing return property are allowed~(Section~\ref{S:approxSM}). In the case of~$\alpha$-supermodular functions, violations are allowed up to a multiplicative factor~$\alpha$, whereas~$\epsilon$-supermodular functions violate supermodularity up to an additive constant~$\epsilon$. We prove that these functions can be minimized near-optimally using greedy search and that~$\alpha$ and~$\epsilon$ not only quantify how much supermodularity is violated, but also the loss in performance guarantee due to these violations~(Theorems~\ref{T:alphaGreedy} and~\ref{T:epsilonGreedy}). We proceed to show that the estimation MSE is~$\alpha$-supermodular and that the worst-case estimation error---i.e., the spectral norm of the error covariance matrix---is~$\epsilon$-supermodular~(Section~\ref{S:approxScalarizations}). Combining these results yields performance guarantees for greedy sensor selection using the MSE and the worst-case error in filtering~(Theorem~\ref{T:nearOptimalFiltering}) and smoothing~(Theorem~\ref{T:nearOptimalSmoothing}) state estimation problems~(Section~\ref{S:nearOptimalSS}). Note that the~$\alpha$-supermodularity and corresponding guarantees for the MSE in filtering and smoothing problems were previously established in~\cite{Chamon17m}. A field monitoring example is provided to illustrate these results~(Section~\ref{S:Sims}).

\textbf{Notation}: Lowercase boldface letters represent vectors~($\bx$), uppercase boldface letters are matrices~($\bX$), and calligraphic letters denote sets~($\calX$). We write~$\abs{\calX}$ for the cardinality of~$\calX$ and let~$\emptyset$ denote the empty set. To say~$\bX$ is a positive semi-definite~(PSD) matrix we write~$\bX \succeq 0$, so that for~$\bX,\bY \in \setR^{n \times n}$, $\bX \preceq \bY \Leftrightarrow \bb^T \bX \bb \leq \bb^T \bY \bb$ for all $\bb \in \setR^n$. Similarly, we write~$\bX \succ 0$ when~$\bX$ is positive definite. For a square matrix~$\bX$, we denote its spectral norm~($\ell_2$ induced norm) by~$\norm{\bX}$ and its maximum and minimum eigenvalues by~$\lmax(\bX)$ and~$\lmin(\bX)$ respectively.

\section{PROBLEM FORMULATION}
\label{S:problem}

Consider an autonomous dynamical system whose internal states~$\bx_k \in \setR^n$ evolve as
\begin{equation}\label{E:dynamics}
	\bx_{k+1} = \bF \bx_k + \bw_k
\end{equation}
and a set of sensors~$\calO$ such that the output~$\by_{u,k} \in \setR^{p_u}$ of sensor~$u \in \calO$ is given by
\begin{equation}\label{E:outputs}
	\by_{u,k} = \bH_u \bx_k + \bv_{u,k}
		\text{,}
\end{equation}
where~$\bF \in \setR^{n \times n}$ is the state transition matrix and~$\bH_u \in \setR^{p_u \times n}$ are the sensor output matrices. The process noise~$\bw_k$ and measurement noises~$\bv_{u,k}$ are zero-mean Gaussian random vectors with arbitrary covariances~$\E \bw_k \bw_k^T = \bR_w \succ 0$ and~$\E \bv_{u,k} \bv_{u,k}^T = \bR_{v,u} \succ 0$ for all~$k$. We assume~$\{\bv_{u_1,i},\bv_{u_2,j},\bw_i,\bw_j\}$ are independent for all~$i \neq j$ and~$u_1,u_2 \in \calO$. The initial states~$\bx_0 \sim \calN(\bar{\bx}_0,\bPi_0)$ are assumed to be a Gaussian random vector with mean~$\bar{\bx}_0$ and covariance~$\bPi_0 \succ 0$.

We seek to estimate the states~$\bx_k$ from the observations~$\{\by_{u,i}\}_{u \in \calX}$ for a sensing set~$\calX \subseteq \calO$ and~$i \leq k$. Depending on the observation history used to estimate the states at time~$k$, we define the following estimation problems:

\begin{enumerate}[(i)]

\item \textbf{Filtering}, where the current states are estimated using past observations, i.e., we seek~$\bxh_k(\calX) = \E \left[ \bx_k \mid \{ \by_{u,i} \}_{u \in \calX,\, i \leq k} \right]$.

\item \textbf{Smoothing}, where all states up to time~$k$ are estimated based on all observations. Formally, we seek~$\bxt_k(\calX) = \E \left[ \bxb_k \mid \{ \by_{u,i} \}_{u \in \calX,\, i \leq k} \right]$, where~$\bxb_k = \vect{cccc}{\bx_{0}^T & \bx_1^T & \cdots & \bx_{k}^T}^T$. As opposed to~(i), states are estimated based on both past and future observations.

\end{enumerate}

Note that we do not consider the problem of \emph{prediction} in which the states at time~$k$ is estimated based on observations up to~$j < k$, i.e., in which we evaluate~$\bxc_k(\calX) = \E \left[ \bx_k \mid \{ \by_{u,i} \}_{u \in \calX,\, i \leq j < k} \right]$. Although filtering and prediction are different estimation problems, they do not differ from the viewpoint of sensor selection. Indeed, prediction only extrapolate the $j$-th filtered estimate using~$\bxc_k = \bF^{k-j} \bxh_j$. Since it does not use additional measurements after the~$j$-th iteration, the sensor selection problem for prediction is in fact the same as the one for filtering.

The performance of the estimation problems~(i) and~(ii) for a given sensing set~$\calX \subseteq \calO$ is fully characterized by the error covariance matrix of the estimators~$\bxh$ and~$\bxt$. Explicitly, 
\begin{subequations}\label{E:covPQ}
\begin{align}
	\bP_k(\calX) &= \E \left[ \bxh_k(\calX) - \bx_k \right] \left[ \bxh_k(\calX) - \bx_k \right]^T
		\label{E:covP}
	\\
	\bQ_k(\calX) &= \E \left[ \bxt_k(\calX) - \bxb_k \right] \left[ \bxt_k(\calX) - \bxb_k \right]^T
		\label{E:covQ}
\end{align}
\end{subequations}
Observe that since the estimators depend on the sensing set~$\calX$, so do their error covariance matrices. In the propositions below, we show that the matrices in~\eqref{E:covPQ} can be written in a unified manner as a set function~$\bY_k: 2^\calO \to \setS_+$ of the form
\begin{equation}\label{E:errorCovariance}
	\bY_k(\calX) = \left( \bM_{\emptyset,k} + \sum_{u \in \calX} \bM_{u,k} \right)^{-1}
		\text{.}
\end{equation}
where~$\bM_{\emptyset,k} \succ 0$ is the \emph{a priori} error covariance matrix that represents the error in the absence of sensor measurements and~$\bM_{u,k} \succeq 0$ carries the information obtained by using sensor~$u \in \calO$ at time~$k$.

\begin{proposition}\label{T:filtering}

The error covariance matrix~$\bP_k(\calX)$ of the filtered estimator~$\bxh$ from~(i) can be written recursively using~\eqref{E:errorCovariance} with
\begin{equation}\label{E:filteringM}
	\bM_{\emptyset,k} = \bP_{k \vert k-1}^{-1}
	\text{ and }
	\bM_{u,k} = \bV_u
		\text{,}
\end{equation}
where~$\bP_{k \vert k-1} = \bF \bP_{k-1} \bF^T + \bR_w$ is the prediction error covariance matrix, $\bP_{k-1}$ is the estimation error covariance matrix at time~$k-1$, $\bP_{0 \vert -1} = \bPi_0$, and~$\bV_u = \bH_u^T \bR_{v,u}^{-1} \bH_u$.

\end{proposition}

\begin{proof}
This result is obtained directly from the Kalman filter~(KF) recursion~(see, e.g., \cite{Kailath00l}). For ease of reference, we provide a derivation in the appendix.
\end{proof}

\begin{proposition}\label{T:smoothing}

The error covariance matrix~$\bQ_k(\calX)$ of the smoothed estimator~$\bxt$ from~(ii) can be written as in~\eqref{E:errorCovariance} with
\begin{equation}\label{E:smoothingM}
\begin{aligned}
	\bM_{\emptyset,k} &= \blkdiag(\bPi_0,\bI \kron \bR_w)^{-1}
	\\
	\bM_{u,k} &= \bPhi_k^T \left( \bI \kron \bV_u \right) \bPhi_k
\end{aligned}
\end{equation}
where~$\bV_u = \bH_u^T \bR_{v,u}^{-1} \bH_u$, $\blkdiag(\bX,\bY)$ is the block diagonal matrix whose diagonal blocks are~$\bX$ and~$\bY$, and
\begin{equation}\label{E:phik}
	\bPhi_k =
	\begin{bmatrix}
		\bI    &&&
		\\
		\bF    & \bI       &&
		\\
		\vdots & \vdots    & \ddots &
		\\
		\bF^k  & \bF^{k-1} & \cdots & \bI
	\end{bmatrix}
		\text{.}
\end{equation}

\end{proposition}

\begin{proof}
The proof follows classical stochastic least-squares results~(see, e.g., \cite{Kailath00l}). For ease of reference, we provide a derivation in the appendix.
\end{proof}

Propositions~\ref{T:filtering} and~\ref{T:smoothing} show how the abstract set function~\eqref{E:errorCovariance} applies to both filtering and smoothing estimation problems. In the sequel, we formulate the sensor selection problem and derive theoretical guarantees directly in terms of the generic~\eqref{E:errorCovariance}. Particular results for the estimation problems~(i) and~(ii) are provided in Section~\ref{S:nearOptimalSS}.

\subsection{Kalman filter sensor selection}
\label{S:KFSS}

Our goal is to study sensor selection problems in which we seek a sensing set~$\calX \subseteq \calO$ constrained by a budget~$\abs{\calX} \leq s$ so that the estimation error covariance matrix~$\bY_k(\calX)$ in~\eqref{E:errorCovariance} is minimized in some sense over a given time horizon~$N$. Formally:

\begin{problem*}[Kalman filter sensor selection]

Take~$N \in \setN$ and let~$\theta_k \geq 0$ for~$k \in [0,N-1]$ be a set of nonnegative scalar weights and~$h: \setS_+ \to \setR$ be a real-valued spectral function over the PSD cone. Kalman filter sensor selection can be written as the discrete optimization problem
\begin{prob}[$\text{P}^\star$]\label{P:sensorSelection}
	\minimize_{\calX \subseteq \calO,\, \abs{\calX} \leq s}& 
		&&\sum_{k = 0}^{N-1} \theta_k\, h\Big[ \bY_{m+k}(\calX) \Big]
			- C_\emptyset
		\text{,}
\end{prob}
where~$C_\emptyset \triangleq \sum_{k = 0}^{N-1} \theta_k h[ \bM_{\emptyset,m+k}^{-1} ]$ is a constant.

\end{problem*}

\noindent It is worth noting that the constant~$C_\emptyset$ is included only so that the objective of~\eqref{P:sensorSelection} vanishes when no sensor is selected---since~$\bY_{k}(\emptyset) = \bM_{\emptyset,k}^{-1}$, taking~$\calX = \emptyset$ in~\eqref{P:sensorSelection} yields~$\sum_{k = 0}^{N-1} \theta_k\,h[ \bY_{m+k}(\emptyset) ] - C_\emptyset = 0$. This normalization has no effect on the optimization problem, but simplifies the statement of the approximation certificates presented in the following sections. Also observe that Propositions~\ref{T:filtering} and~\ref{T:smoothing} do not take control inputs into account. Hence, though a sensing set obtained from~\eqref{P:sensorSelection} can be used to perform state estimation in the presence of control inputs, it need not be optimal for the joint sensing set-controller design problem. This formulation is left for future works.

Different choices of~$N$, $\theta_k$, and~$h$ in~\eqref{P:sensorSelection} yield different sensor selection problems. For instance, \eqref{P:sensorSelection} becomes a myopic sensor selection problem when~$N = 1$~\cite{Shamaiah10g}. When~$N > 1$, we can choose to optimize the \emph{final estimation error}---corresponding to~$\theta_k = 0$ for all~$k < N-1$ and~$\theta_{N-1} = 1$---, \emph{the average error}---by making~$\theta_k = 1$ for all~$k$---, or some \emph{weighted average error}---e.g., a geometric discount using~$\theta_k = \rho^{N-1-k}$ for~$\rho < 1$. Observe that since the smoothing problem~(ii) estimates all states up to~$k$, these multi-step versions of~\eqref{P:sensorSelection} may not be useful in practice. We still consider them for the sake of symmetry.

The scalarization~$h$ is necessary because matrix minimization in the PSD cone is not well-posed~\cite{Boyd04c} and its choice affects both the complexity and the practical usefulness of~\eqref{P:sensorSelection}. For instance, \eqref{P:sensorSelection} with~$h(\bX) = -\trace(\bX^{-1})$ is combinatorial but modular and can therefore be solved in linear time~\cite{Krause14s, Bach14l}. This objective, however, is of little relevance in practice inasmuch as it has no direct relation to the estimation MSE. In general, three scalarizations are widely agreed to be of practical significance:

\begin{enumerate}[i.]

\item The \textbf{trace} defined as
\begin{prob}\label{P:trace}
	\calT^\star \in \argmin_{\calT \subseteq \calO,\,\abs{\calT} \leq s}
		f_T(\calT) \triangleq \sum_{k = 0}^{N-1} \theta_k \trace\Big[ \bY_{m+k}(\calT) \Big]
			- C_\emptyset
		\text{,}
\end{prob}
in which case the objective represents the estimation MSE.

\item The \textbf{spectral norm} or the maximum eigenvalue
\begin{prob}\label{P:eig}
	\calS^\star \in \argmin_{\calS \subseteq \calO,\,\abs{\calS} \leq s}
		f_S(\calS) \triangleq \sum_{k = 0}^{N-1} \theta_k \norm{\bY_{m+k}(\calS)}
			- C_\emptyset
		\text{,}
\end{prob}
in which case the objective is a robust version of~\eqref{P:trace}, since it minimizes the worst-case estimation error.

\item The \textbf{log-determinant}
\begin{prob}\label{P:logdet}
	\calD^\star \in \argmin_{\calD \subseteq \calO,\,\abs{\calD} \leq s}
		\sum_{k = 0}^{N-1} \theta_k \log\det\Big[\bY_{m+k}(\calD)\Big]
			- C_\emptyset
		\text{,}
\end{prob}
in which case the objective corresponds to the volume of the confidence ellipsoid~\cite{Joshi09s, Tzoumas16s}.
\end{enumerate}

These formulations are natural because they minimize different estimation error metrics related to the MSE, which is a customary objective in traditional state estimation, i.e., problems that do not consider sensor selection~($\calX = \calO$). Still, problems~\eqref{P:trace}--\eqref{P:logdet} are challenging to solve. They are known to be NP-hard~\cite{Krause08n, Krause14s, Ranieri14n, Tzoumas16m, Olshevsky17o, Zhang17s, Summers16s}, so their solutions can only be approximated in practice, typically by means of greedy algorithms~\cite{Summers16s, Zhang17s}.

In the case of problem~\eqref{P:logdet}, the log-determinant is a monotonically decreasing and supermodular set function~\cite{Shamaiah10g, Tzoumas16s, Summers16s, Moshtagh09o}, i.e., it displays a ``diminishing returns'' property that yields near-optimal performance guarantees for its greedy minimization~\cite{Nemhauser78a}. The MSE and the worst-case estimation error, on the other hand, are not supermodular in general~\cite{Olshevsky17o, Tzoumas16m, Krause14s, Ranieri14n, Zhang17s}. In fact, stringent conditions on the system matrices are needed to make the estimation MSE supermodular~\cite{Das08a, Singh17s}. These facts notwithstanding, greedy sensor selection has been observed to perform well when solving \eqref{P:trace} and \eqref{P:eig} in both control~\cite{Summers16s, Tzoumas16m, Zhang17s} and other contexts~\cite{Avron13f, Chamon17a}.

Our goal in the remainder of this paper is to show that in certain noise regimes of interest these objectives are approximately submodular, so that the greedy solutions of~\eqref{P:trace} and~\eqref{P:eig} come with performance guarantees analogous to those of~\eqref{P:logdet}, thereby reconciling empirical observations with our theoretical understanding of these problems. To obtain these near-optimal certificates, the following section develops a theory of \emph{approximately supermodular} functions, studying operations that preserve approximate supermodularity and showing that these functions can be near-optimally minimized. We then show that the trace and the spectral norm are approximately supermodular scalarizations of~$\bY_k$ so as to finally provide explicit suboptimality bounds for greedy sensor selection for filtering/smoothing.

\begin{remark}\label{R:logdet}
Besides its intrinsic value as a minimizer of the volume of the confidence ellipsoid \cite{Joshi09s, Tzoumas16s}, the~$\log\det$ is often put forward as a supermodular surrogate for the MSE. It is important to point out that this figure of merit is not directly related to the MSE or the worst-case error and that, in general, it is not a good surrogate for either of them. In fact, the log determinant is a good substitute for the MSE only when the problem has some inherent structure that constraints the confidence ellipsoid to be somewhat symmetric. Otherwise, since the volume of an ellipsoid can be reduced by decreasing the length of a single principal axis, using the~$\log\det$ can lead to designs that perform well---in the MSE sense---along a few directions of the parameter space and poorly along all others. Formally, this can be seen by comparing the variation of the log-determinant and trace functions with respect to the eigenvalues of the PSD matrix~$\bY$,
\begin{equation*}
	\frac{\del \log\det(\bY)}{\del \lambda_j(\bY)} = \frac{1}{\lambda_j(\bY)}
	\qquad \text{and} \qquad
	\frac{\del \trace(\bY)}{\del \lambda_j(\bY)} = 1
		\text{.}
\end{equation*}
The gradient of the log-determinant is largest in the direction of the smallest eigenvalue of the error covariance matrix. In contrast, the MSE gives equal weight to all directions of the space. The latter yields balanced designs that are similar to the former only if those are forced to be balanced by the problem structure. Even for random dynamical systems with~$10$~states and~$10$ outputs~(see Section~\ref{S:Sims} for details), the MSE error achieved by a sensing set~($s = 3$) selected by minimizing the~$\log\det$ was up to~$13\%$ larger than a sensing set selected by minimizing the MSE itself~($100$ realizations).
\end{remark}

\section{APPROXIMATELY SUPERMODULAR FUNCTION MINIMIZATION}
	\label{S:approxSM}

\emph{Supermodularity}~(\emph{submodularity}) encodes a ``diminishing returns'' property of certain set functions that implies near-optimality bounds on their greedy minimization~(maximization). Well-known representatives of this class include the rank or~$\log\det$ of a sum of PSD matrices, the Shannon entropy, and the mutual information~\cite{Krause14s, Bach14l}. Still, supermodularity is a stringent condition. In particular, it does not hold for the objectives of~\eqref{P:trace} or~\eqref{P:eig}~\cite{Olshevsky17o, Tzoumas16m, Krause14s, Ranieri14n}.

The purpose of \emph{approximate supermodularity}~(\emph{submodularity}) is to relax the original ``diminishing returns'' property while controlling the magnitude of the violations. The rationale is that if a function is ``almost'' supermodular, then it should behave similar to a supermodular function. In what follows, we formalize and quantify these statements.

Consider a set function~$f: 2^\calO \to \setR$ whose value for an arbitrary set~$\calX \subseteq \calO$ is denoted by~$f(\calX)$. We say~$f$ is \emph{normalized} if~$f(\emptyset) = 0$ and~$f$ is \emph{monotone decreasing} if for all sets~$\calA \subseteq \calB \subseteq \calO$ it holds that~$f(\calA) \geq f(\calB)$. Observe that if a function is normalized and monotone decreasing it must be that~$f(\calX) \leq 0$ for all~$\calX \subseteq \calO$. Define
\begin{equation}\label{E:incrementalGain}
	\Delta_u f(\calX) = f\left( \calX \right)
		- f\left( \calX \cup \{u\} \right)
\end{equation}
to be the variation in the value of~$f$ incurred by adding the element~$u \in \calO \setminus \calX$ to the set~$\calX$. Then, a set function~$f$ is \emph{supermodular} if for all sets~$\calA \subseteq \calB \subseteq \calO$ and elements~$u \in \calO \setminus \calB$ it holds that
\begin{equation}\label{E:supermodularity}
	\Delta_u f(\calA) \geq \Delta_u f(\calB)
		\text{.}
\end{equation}
A function~$f$ is \emph{submodular} if~$-f$ is supermodular.

The relevance of supermodular functions in this work is due to the celebrated bound on the suboptimality of their greedy minimization~\cite{Nemhauser78a}. Specifically, consider the generic cardinality constrained optimization problem
\begin{prob}\label{P:cardConstraint}
	\calX^\star \in \argmin_{\abs{\calX} \leq s} f(\calX)
		\text{,}
\end{prob}
and construct its greedy solution by starting with~$\calG_0 = \emptyset$ and incorporating the elements from~$\calO$ one at a time so as to maximize the gain at each step. Explicitly, at step~$j$ we do
\begin{equation}\label{E:greedy}
\begin{gathered}
	\calG_{j+1} = \calG_{j} \cup \{u\}
		\text{,}
	\\
	\text{with} \quad
	u = \argmin_{w \in \calO \setminus \calG_{j}}
		f\left( \calG_{j} \cup \{w\} \right)
		\text{.}
\end{gathered}
\end{equation}
The recursion in~\eqref{E:greedy} is repeated for~$s$ steps to obtain a greedy solution with~$s$ elements. If~$f$ is monotone decreasing and supermodular~\cite{Nemhauser78a}, then
\begin{align}\label{E:greedyNWF}
	f(\calG_s) \leq (1-e^{-1}) f(\calX^\star)
		\text{.}
\end{align} 

The guarantee in~\eqref{E:greedyNWF}, however, no longer applies when~$f$ is not supermodular. To provide guarantees in these cases, we leverage two measures of approximate supermodularity and derive near-optimality bounds for each of them. It is worth noting that though intuitive, such results are not trivial. In fact, \cite{Horel16m} showed that for another measure of proximity, functions~$\delta$-close to supermodular cannot be optimized in polynomial time unless~$\delta$ is small.

We start with a multiplicative relaxation of the supermodular property~\eqref{E:supermodularity}.

%%%%%%%%%%%%%%%%%%%%%%%%%%%%%%%%%%%%%%%%%%%%%%%%%%
%%%%% ALPHA-SUPERMODULARITY
%%%%%%%%%%%%%%%%%%%%%%%%%%%%%%%%%%%%%%%%%%%%%%%%%%
\begin{definition}[$\alpha$-supermodularity]
	\label{D:alphaSM}

A set function $f: 2^\calO \to \setR$ is \emph{$\alpha$-supermodular}, for~$\alpha \in \setR$, if for all sets~$\calA \subseteq \calB \subseteq \calO$ and all~$u \in \calO \setminus \calB$ it holds that
\begin{equation}\label{E:alphaSM}
	\Delta_u f(\calA) \geq \alpha \, \Delta_u f(\calB)
		\text{.}
\end{equation}
\end{definition}

For~$\alpha \geq 1$, \eqref{E:alphaSM} reduces the original definition of supermodularity~\eqref{E:supermodularity}, in which case we refer to the function simply as supermodular~\cite{Bach14l, Krause14s}. On the other hand, when~$\alpha < 1$, $f$ is said to be \emph{approximately supermodular}. Notice that if~$f$ is decreasing, then~\eqref{E:alphaSM} always holds for~$\alpha = 0$. We are therefore interested in the largest~$\alpha$ for which~\eqref{E:alphaSM} holds, i.e.,
\begin{equation}\label{E:alpha}
	\alpha =
	\min_{\substack{\calA \subset \calB \subseteq \calO \\
	u \in \calO \setminus \calB}}\ 
		\frac{\Delta_u f(\calA)}{\Delta_u f(\calB)}
\end{equation}

This concept first appeared in the context of auction design~\cite{Lehmann06c}, although it has been rediscovered in the context of discrete optimization, estimation, and control~\cite{Sviridenko14o, Chamon16n, Chamon17m, Bian17g}. It is worth noting that~$\alpha$ in Definition~\ref{D:alphaSM} is also related to the \emph{submodularity ratio}~$\gamma$ introduced in~\cite{Das11s}. However, the proposed bounds on~$\gamma$ depended on the sparse eigenvalues of a matrix, that are NP-hard to compute. The first explicit~(P-computable) bounds on~$\alpha$ were obtained in~\cite{Chamon16n, Chamon17m}, though the same bounds were more recently derived for~$\gamma$ for a special case of~\eqref{P:trace}~($m = 1$, $\bM_\emptyset = \sigma^2 \bI$, and rank-one~$\bM_u$)~\cite{Bian17g}. These results were particularized for the problem of controlling linear systems in~\cite{Summers19p}.

Interestingly, $\alpha$ not only measures how much~$f$ violates supermodularity, but it also quantifies the loss in performance guarantee incurred from these violations.

\begin{figure}[tb]
	\centering
	\includesvg{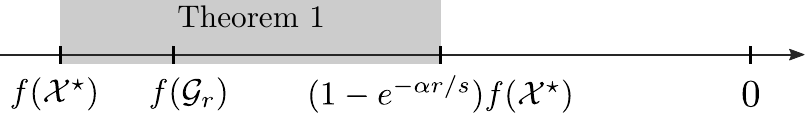}
	\caption{Illustration of the near-optimal guarantee from Theorem~\ref{T:alphaGreedy}.}
	\label{F:diagram}
\end{figure}

%%%%%%%%%%%%%%%%%%%%%%%%%%%%%%%%%%%%%%%%%%%%%%%%%%
%%%%% GREEDY FOR ALPHA-SUPERMODULAR
%%%%%%%%%%%%%%%%%%%%%%%%%%%%%%%%%%%%%%%%%%%%%%%%%%
\begin{theorem}\label{T:alphaGreedy}

	Let~$f$ be a normalized, monotone decreasing, and~$\alpha$-supermodular set function~(i.e., $f(\calX) \leq 0$ for all~$\calX \subseteq \calO$). Then, the solution obtained after~$r$ steps of the greedy algorithm in~\eqref{E:greedy} obeys
	\begin{equation}\label{E:alphaRelOpt}
		f(\calG_r) \leq (1 - e^{-\alpha r/s}) f(\calX^\star)
			\text{.}
	\end{equation}
\end{theorem}

\begin{proof}
	See appendix.
\end{proof}

Theorem~\ref{T:alphaGreedy} bounds the suboptimality of the greedy solution from~\eqref{E:greedy} when its objective is~$\alpha$-supermodular. Indeed, since~$f$ is a non-positive function, it guarantees that~$f(\calG_r)$ cannot be too far from the optimal value~(Figure~\ref{F:diagram}). At the same time, it quantifies the effect of relaxing the supermodularity hypothesis typically used to provide performance guarantees in these settings. In fact, if~$f$ is supermodular~($\alpha = 1$) and~$r = s$, we recover the guarantee~\eqref{E:greedyNWF} from~\cite{Nemhauser78a}. On the other hand, for an approximately supermodular function~($\alpha < 1$), the result in~\eqref{E:alphaRelOpt} shows that the same~$63\%$ guarantee is recovered by greedily selecting a set of size~$s/\alpha$. Hence, $\alpha$ not only measures how much~$f$ violates supermodularity, but also gives a factor by which a solution set must increase to maintain supermodular-like near-optimality. It is worth noting that, as with the original bound in~\cite{Nemhauser78a}, \eqref{E:alphaRelOpt} is not tight and that better results are typically obtained in practice~(see Section~\ref{S:Sims}).

Although $\alpha$-supermodularity yields a multiplicative approximation factor, finding meaningful bounds on~$\alpha$ can be challenging for certain set functions, such as the objective of~\eqref{P:eig}. It is therefore useful to look at approximate supermodularity from a different perspective, as proposed in~\cite{Krause10s}.

%%%%%%%%%%%%%%%%%%%%%%%%%%%%%%%%%%%%%%%%%%%%%%%%%%
%%%%% EPSILON-SUPERMODULARITY
%%%%%%%%%%%%%%%%%%%%%%%%%%%%%%%%%%%%%%%%%%%%%%%%%%
\begin{definition}[$\epsilon$-supermodularity]

A set function~$f: 2^\calO \to \setR$ is \emph{$\epsilon$-supermodular}, for~$\epsilon \in \setR$, if for all multisets~$\calA \subseteq \calB \subseteq \calO$ and all~$u \in \calO \setminus \calB$ it holds that
\begin{equation}\label{E:epsilonSM}
	\Delta_u f(\calA) \geq \Delta_u f(\calB) - \epsilon
		\text{.}
\end{equation}

\end{definition}

Again, we say~$f$ is supermodular if~$\epsilon \leq 0$ and approximately supermodular otherwise. As with~$\alpha$, we want the best~$\epsilon$ that satisfies~\eqref{E:epsilonSM}, which is given by
\begin{equation}\label{E:epsilon}
	\epsilon =
	\max_{\substack{\calA \subseteq \calB \subseteq \calO \\
		u \in \calO \setminus \calB}}\ 
	\Delta_u f(\calB) - \Delta_u f(\calA)
		\text{.}
\end{equation}
In contrast to~$\alpha$-supermodularity, we obtain an additive approximation guarantee for the greedy minimization of $\epsilon$-supermodular functions.

%%%%%%%%%%%%%%%%%%%%%%%%%%%%%%%%%%%%%%%%%%%%%%%%%%
%%%%% GREEDY FOR EPSILON-SUPERMODULAR
%%%%%%%%%%%%%%%%%%%%%%%%%%%%%%%%%%%%%%%%%%%%%%%%%%
\begin{theorem}\label{T:epsilonGreedy}

	Let~$f$ be a normalized, monotone decreasing, and~$\epsilon$-supermodular set function~(i.e., $f(\calX) \leq 0$ for all~$\calX \subseteq \calO$). Then, the solution obtained after~$r$ steps of the greedy algorithm in~\eqref{E:greedy} obeys
	\begin{equation}\label{E:epsilonRelOpt}
		f(\calG_r) \leq (1-e^{-r/s}) \left[ f(\calX^\star) + s \cdot \epsilon \right]
			\text{.}
	\end{equation}
\end{theorem}

\begin{proof}
	See appendix.
\end{proof}

As before, $\epsilon$ quantifies the loss in performance guarantee due to relaxing supermodularity. Indeed, \eqref{E:epsilonRelOpt} reveals that~$\epsilon$-supermodular functions have the same guarantees as a supermodular function up to an additive factor of~$\Theta(s \epsilon)$. In fact, if~$\epsilon \leq (e s)^{-1} \abs{f(\calX^\star)}$~(recall that~$f(\calX^\star) \leq 0$ due to normalization), then taking~$r = 3s$ recovers the supermodular~$63\%$ approximation factor. This same factor is obtained for~$(\alpha \geq 1/3)$-supermodular functions.

Although Theorems~\ref{T:alphaGreedy} and~\ref{T:epsilonGreedy} characterize the loss in suboptimality incurred from violating supermodularity, their performance certificates depend on the specific values of~$\alpha$ and~$\epsilon$. However, \eqref{E:alpha} and~\eqref{E:epsilon} reveal that finding~$\alpha$ and~$\epsilon$ for a general function is a combinatorial problem. To give actual near-optimal guarantees for greedy solution of the sensor selection problems~\eqref{P:trace} and~\eqref{P:eig}, the next section bounds the values of~$\alpha$ and~$\epsilon$ for different scalarizations of~\eqref{E:errorCovariance}. Notice that we do not need to tackle the objectives of~\eqref{P:trace} and~\eqref{P:eig} directly thanks to the following lemma.

\begin{lemma}\label{T:preserveSM}
Consider the set functions~$f_i: 2^\calO \to \setR$, $i \in \setN$. Then, for~$\theta_i \geq 0$ and~$b \in \setR$,
\begin{enumerate}[(i)]

\item if the~$f_i$ are $\alpha_i$-supermodular, then $g = \sum_i \theta_i f_i + b$ is $\min(\alpha_i)$-supermodular;

\item if the~$f_i$ are $\epsilon_i$-supermodular, then $g = \sum_i \theta_i f_i + b$ is $\left( \sum_i \theta_i \epsilon_i \right)$-supermodular.

\end{enumerate}
\end{lemma}

\begin{proof}
See appendix.
\end{proof}

\section{APPROXIMATELY SUPERMODULAR SCALARIZATIONS}
	\label{S:approxScalarizations}

Theorems~\ref{T:alphaGreedy} and~\ref{T:epsilonGreedy} apply to set functions that are (i)~normalized, (ii)~monotone decreasing, and (iii)~approximately supermodular. By construction, the objective of~\eqref{P:sensorSelection} is normalized~[(i)]. In this section, we obtain properties~(ii) and~(iii) by leveraging Lemma~\ref{T:preserveSM} and focusing only on the scalarizations underlying the objectives of~\eqref{P:trace} and~\eqref{P:eig}, namely the trace
\begin{equation}\label{E:traceY}
	t(\calX) = \trace[\bY_k(\calX)]
\end{equation}
and the spectral norm
\begin{equation}\label{E:normY}
	e(\calX) = \norm{\bY_k(\calX)}
		\text{,}
\end{equation}
for~$\bY_k$ defined as in~\eqref{E:errorCovariance}.

To establish that set functions~$t$ and~$e$ are monotone decreasing~[(ii)], we prove that~$\bY_k(\calX)$ is a decreasing set function in the PSD cone~(Lemma~\ref{T:monotonicity}). The definition of Loewner order and the monotonicity of the trace~\cite{Horn13} imply the desired result. The monotone decreasing property of the objectives of~\eqref{P:trace} and~\eqref{P:eig} then follows immediately from~$\theta_j \geq 0$.

\begin{lemma}\label{T:monotonicity}

The matrix-valued set function~$\bY_k(\calX)$ in~\eqref{E:errorCovariance} is monotonically decreasing with respect to the PSD cone, i.e., $\calA \subseteq \calB \subseteq \calO \Rightarrow \bY_k(\calA) \succeq \bY_k(\calB)$.

\end{lemma}

\begin{proof}

The monotone decreasing nature of~$\bY_k$ stems from the fact that matrix inversion is an operator antitone function, i.e., that for~$\bX,\bY \succeq 0$, it holds that~$\bX \succeq \bY \Leftrightarrow \bX^{-1} \preceq \bY^{-1}$~\cite{Bhatia97m}. To see this, write~\eqref{E:errorCovariance} as~$\bY_k(\calX) = \bR(\calX)^{-1}$, where~$\bR(\calX) = \bM_{\emptyset,k} + \sum_{u \in \calX} \bM_{u,k}$. Then, notice that since~$\bR$ is a sum of PSD matrices, it holds that~$\bR(\calA) \succeq 0$ for all~$\calA \subseteq \calO$. Moreover, $\bR$ is a modular~(additive) function, i.e., $\bR(\calA \cup \calB) = \bR(\calA) + \bR(\calB)$. Hence, for~$\calA \subseteq \calB \subseteq \calO$, we obtain
\begin{equation}\label{E:increasingR}
	\bR(\calB) = \bR(\calA) + \bR(\calB \setminus \calA)
		\succeq \bR(\calA)
		\text{.}
\end{equation}
It is straightforward to obtain from~\eqref{E:increasingR} that for~$\calA \subseteq \calB \subseteq \calO$
\begin{equation*}
	\bR(\calA) \preceq \bR(\calB)
	\Leftrightarrow
	\bR(\calA)^{-1} \succeq \bR(\calB)^{-1}
	\Leftrightarrow
	\bY_k(\calA) \succeq \bY_k(\calB)
		\text{.}
\end{equation*}
\end{proof}

The remainder of this section is dedicated to showing that the trace and the spectral norm are approximately supermodular scalarizations of~$\bY_k$, starting with the set trace function~\eqref{E:traceY}. In what follows, we omit the dependence of~$\bY$, $\bM_\emptyset$, and~$\bM_u$ on~$k$ for clarity.

\begin{theorem}\label{T:alphaBound}

Let~$t$ be the set function in~\eqref{E:traceY}. Then, $t$ is~$\alpha$-supermodular with
\begin{equation}\label{E:alphaBound}
	\alpha \geq \frac{\mu_\textup{min}}{\mu_\textup{max}} > 0
		\text{,}
\end{equation}
where
\begin{equation*}
	0 < \mu_\textup{min} \leq \lmin\left[ \bM_{\emptyset} \right] \leq
		\lmax\left[ \bM_{\emptyset} + \sum_{u \in \calO} \bM_u \right]
			\leq \mu_\textup{max}
			\text{.}
\end{equation*}

\end{theorem}

\begin{remark}

Although there exist examples for which~$t(\calX)$ is not supermodular, the general statement of Theorem~\ref{T:alphaBound} does not allow us to claim that~$\alpha < 1$. A simple counter-example involves the case in which~$\bM_{\emptyset} = \mu_0 \bI$ and~$\bM_u = \mu_u \bI$, $\mu_0,\mu_u \geq 0$, so that~$t$ becomes effectively a scalar function of~$\mu_0,\mu_u$. Since scalar convex functions of positive modular functions are supermodular~\cite{Lovasz82s}, we have~$\alpha \geq 1$ in this case.

\end{remark}

The proof of Theorem~\ref{T:alphaBound} relies on the following bounds on the variation~$\Delta_u t(\calX)$:
\begin{lemma}\label{T:DeltaLemma}

For all~$\calX \subset \calO$ and~$u \in \calO \setminus \calX$, it holds that
\vspace*{-5pt}
\begin{multline}\label{E:DeltaBound}
	\lmin\left[ \bY(\calX) \right]
	\trace \left[ \bM_u \bY\left( \calX \cup \{u\} \right) \right]
	\leq
	\Delta_u t(\calX) \leq{}
	\\
	\lmax\left[ \bY(\calX) \right]
	\trace \left[ \bM_u \bY\left( \calX \cup \{u\} \right) \right]
		\text{.}
\end{multline}
\end{lemma}

\begin{proof}
See appendix.
\end{proof}

Theorem~\ref{T:alphaBound} then follows readily from this technical lemma.

\begin{proof}[Proof of Theorem~\ref{T:alphaBound}]
Notice from~\eqref{E:incrementalGain} and~\eqref{E:alpha} that~$\alpha$ can be written as
\begin{equation*}
	\alpha = \min_{\substack{\calA \subset \calB \subseteq \calO \\ u \in \calO \setminus \calB}}
		\frac{\Delta_u t(\calA)}{\Delta_u t(\calB)}
		\text{.}
\end{equation*}
From Lemma~\ref{T:DeltaLemma}, we then obtain
\begin{align}
	\alpha &\geq
	\frac{
		\lmin\left[ \bY(\calA) \right]
	}{
		\lmax\left[ \bY(\calB) \right]
	}
	\times
	\frac{
		\trace\left[
			\bM_u \bY\left( \calA \cup \{u\} \right)
		\right]
	}{
		\trace\left[
			\bM_u \bY\left( \calB \cup \{u\} \right)
		\right]
	}
		\text{.}
		\label{E:preAlphaBound2}
\end{align}
To simplify~\eqref{E:preAlphaBound2}, recall from Lemma~\ref{T:monotonicity} that~$\bY$ is monotone decreasing in the PSD cone. Since~$\calA \subset \calB \subseteq \calO$, it holds that~$\bY\left( \calA \cup \{u\} \right) \succeq \bY\left( \calB \cup \{u\} \right)$. The ordering of the PSD cone~(Loewner order) gives us that the second term in~\eqref{E:preAlphaBound2} is always greater than one, which yields
\begin{equation*}
	\alpha \geq
	\frac{
		\lmin\left[ \bY(\calA) \right]
	}{
		\lmax\left[ \bY(\calB) \right]
	} 
		\text{.}
\end{equation*}
The lower bound in~\eqref{E:alphaBound} is readily obtained by observing that the decreasing nature of~$\bY$ implies that for any set~$\calX \subseteq \calO$:
\begin{equation*}
	\lmin\left[ \bY(\calO) \right] \leq
	\lmin\left[ \bY(\calX) \right] \leq
	\lmax\left[ \bY(\calX) \right] \leq
	\lmax\left[ \bY(\emptyset) \right]
		\text{.}
\end{equation*}
\end{proof}

Theorem~\ref{T:alphaBound} gives a deceptively simple bound on the~$\alpha$-supermodularity of the set function~$t$ in~\eqref{E:traceY} depending on the spectrum of the~$\bM_\emptyset,\bM_i$. This bound can be interpreted geometrically in terms of the ``range'' of the error covariance matrix~$\bY$. To see this, define the \emph{numerical range} of the set function~$\bY$ as
\begin{equation}\label{E:W}
	W_\calO(\bY) =
		W \left[ \bigoplus_{\calX \subseteq \calO} \bY(\calX) \right]
		\text{,}
\end{equation}
where~$\bA \oplus \bB = \blkdiag(\bA,\bB)$ is the direct sum of matrices~$\bA$ and~$\bB$ and~$W(\bM) = \{ \bx^T \bM \bx \mid \norm{\bx}_2 = 1 \}$ is the classical numerical range~\cite{Horn13}. Since the numerical range is a convex set, we can define its relative diameter as
\begin{equation}\label{E:delta}
	\Delta = \max_{\mu,\eta \in W_\calO(\bY)}
		\abs{\frac{\mu - \eta}{\mu}}
		\text{.}
\end{equation}
Then, the following holds:

\begin{proposition}

The set functions~$t$ in~\eqref{E:traceY} is~$\alpha$-supermodular with
\begin{equation*}
	\alpha \geq 1 - \Delta
		\text{,}
\end{equation*}
where~$\Delta$ is the relative diameter of the numerical range of~$\bY$ in~\eqref{E:delta}.

\end{proposition}

\begin{proof}
Since~$\bY \succ 0$, the numerical range in~\eqref{E:W} is the bounded convex hull of the eigenvalues of~$\bY(\calX)$ for all~$\calX \subseteq \calO$~\cite{Horn13}. We can therefore simplify~\eqref{E:delta} using the fact that it is monotonically increasing in~$\mu$ and decreasing in~$\eta$. Explicitly,
\begin{equation*}
	\Delta = \max_{\calX,\calY \subseteq \calO}
	\abs{
		\frac{
			\lmax\left[ \bY(\calY) \right]
			- \lmin\left[ \bY(\calX) \right]
		}{
			\lmax\left[ \bY(\calY) \right]
		}
	}
		\text{.}
\end{equation*}
Using the fact that~$\bY$ is monotonically decreasing~(Lemma~\ref{T:monotonicity}), this maximum is achieved for
\begin{equation*}
	\Delta = \frac{
		\lmax\left[ \bY(\emptyset) \right]
		- \lmin\left[ \bY(\calO) \right]
	}{
		\lmax\left[ \bY(\emptyset) \right]
	}
		\text{.}
\end{equation*}
The bound in~\eqref{E:alphaBound} thus becomes
\begin{equation*}
	\alpha \geq
	\frac{
		\lmin\left[ \bY(\calO) \right]
	}{
		\lmax\left[ \bY(\emptyset) \right]
	} = 1 - \Delta
		\text{.}\qedhere
\end{equation*}
\end{proof}

Hence, \eqref{E:alphaBound} bounds how much~$t$ deviates from a supermodular function~(as quantified by~$\alpha$) in terms of the numerical range of its underlying function~$\bY$. The shorter the range of~$\bY$, the more supermodular-like~\eqref{E:traceY} will be.

Notice from Lemma~\ref{T:preserveSM}, that evaluating the $\alpha$-supermodularity of the objective of~\eqref{P:trace} is a straightforward corollary of Theorem~\ref{T:alphaBound}.

\begin{corollary}\label{T:alphaBoundCombination}

The objective of~\eqref{P:trace} is $\alpha$-supermodular for
\begin{equation}\label{E:alphaTrace}
	\alpha \geq \min_{0 \leq k \leq N-1}
		\frac{
			\lmin\left[ \bM_{\emptyset,m+k} \right]
		}{
			\lmax\left[ \bM_{\emptyset,m+k} + \sum_{u \in \calO} \bM_{u,,m+k} \right]
		}
		\text{.}
\end{equation}

\end{corollary}

Proceeding, we now bound the~$\epsilon$-supermodularity of~\eqref{E:normY}, once again temporarily omitting the dependences on~$k$.

\begin{theorem}\label{T:epsilonBound}

Let~$e$ be defined as in~\eqref{E:normY}. Then, $e$ is $\epsilon$-supermodular with
\begin{equation}\label{E:epsilonBoundE}
	\epsilon \leq
	\frac{
		\lmax(\sum_{u \in \calO} \bM_u)
	}{
		\lmin\left[ \bM_\emptyset \right]^{2}
	}
	\text{.}
\end{equation}
\end{theorem}

\begin{proof}

The proof follows a homotopy argument, i.e., we define a continuous map between~$\Delta_u e(\calA)$ and~$\Delta_u e(\calB)$ and bound its derivative using spectral bounds on the sum of Hermitian matrices. The inequality in~\eqref{E:epsilonBoundE} then follows from the fundamental theorem of calculus.

Start by defining the homotopy, with~$t \in [0,1]$,
\begin{equation}\label{E:homotopyE}
	h_{\calA\calB}(t) = \norm{\bZ(t)^{-1}}
		- \norm{\left( \bZ(t) + \bM_u \right)^{-1}}
\end{equation}
for~$\calA \subset \calB \subseteq \calO$, $\bZ(t) = \bR(\calA) + t\left[ \bR(\calB) - \bR(\calA) \right]$, and~$\bR(\calX) = \bM_{\emptyset} + \sum_{u \in \calX} \bM_{u}$. Note that since~$\bY(\calX) = \bR(\calX)^{-1}$, we have~$h_{\calA\calB}(0) = \Delta_u e(\calA)$ and~$h_{\calA\calB}(1) = \Delta_u e(\calB)$. If~$\dot{h}_{\calA\calB}(t)$ is the derivative of~$h_{\calA\calB}$ with respect to~$t$, it therefore holds that
\begin{equation}\label{E:calculus}
	\Delta_u e(\calB) = \Delta_u e(\calA)
		+ \int_{0}^{1} \dot{h}_{\calA\calB}(t) dt
		\text{.}
\end{equation}
Comparing~\eqref{E:calculus} to the definition of~$\epsilon$ in~\eqref{E:epsilon}, we obtain
\begin{equation}\label{E:epsilonHomotopyE}
	\epsilon =
		\max_{\substack{\calA \subset \calB \subseteq \calO \\
			u \in \calO \setminus \calB}}
		\int_{0}^{1} \dot{h}_{\calA\calB}(t) dt
		\text{.}
\end{equation}

We now proceed by evaluating~$\dot{h}$. We omit the dependence on~$\calA$ and~$\calB$ for conciseness. First, recall from matrix analysis that for any~$\bX \succ 0$ we have
\begin{equation}\label{E:diff}
\begin{aligned}
	\frac{d}{dt} \norm{\bX(t)^{-1}} &= 
	\bq(t)^T \left[ \frac{d}{dt} \bX(t)^{-1} \right] \bq(t)
	\\
	{}&= -\bq(t)^T \bX(t)^{-1} \bXd(t) \bX(t)^{-1} \bq(t)
		\text{,}
\end{aligned}
\end{equation}
where~$\bq(t)$ is the eigenvector relative to the maximum eigenvalue of~$\bX(t)$. To obtain~\eqref{E:diff}, we used the fact that~$\norm{\bX} = \lmax(\bX)$ for~$\bX \succeq 0$ and~$\frac{d}{dt} \lmax[\bX(t)] =\bq(t)^T \bXd(t) \bq(t)$. We then used~$\frac{d}{dt} \bX(t)^{-1} = - \bX(t)^{-1} \bXd(t) \bX(t)^{-1}$~\cite{Bhatia97m}. In view of~\eqref{E:homotopyE} and~\eqref{E:diff}, we obtain
\begin{equation}\label{E:homotopyEDiff}
\begin{aligned}
	\dot{h}(t) &=
	\bw(t)^T \left[ \bZ(t) + \bM_u \right]^{-1}
	\left[ \bR(\calB) - \bR(\calA) \vphantom{x^2} \right]
	\\
	&\qquad\quad {}\times \left[ \bZ(t) + \bM_u \right]^{-1} \bw(t)
	\\
	{}&- \bu(t)^T \bZ(t)^{-1}
	\left[ \bR(\calB) - \bR(\calA) \vphantom{x^2} \right]
	\bZ(t)^{-1} \bu(t)
		\text{,}
\end{aligned}
\end{equation}
where~$\bu(t)$~and $\bw(t)$ are the eigenvectors relative to the maximum eigenvalues of~$\bZ(t)^{-1}$ and~$\left[ \bZ(t) + \bM_u \right]^{-1}$ respectively.

We now proceed by finding a bound for~$\dot{h}$ in~\eqref{E:homotopyEDiff} that is independent of~$t$. To do so, observe that~$\bR(\calB) \succeq \bR(\calA)$, since~$\calA \subseteq \calB$ and~$\bR$ is monotone increasing~(Lemma~\ref{T:monotonicity}). By the Loewner ordering, the second term in~\eqref{E:homotopyEDiff} is therefore non-positive. Thus, using Rayleigh's spectral inequality and the fact~$\norm{\bw(t)} = 1$ for all~$t$ yields
\begin{equation*}
	\dot{h}(t) \leq
	\left\| \left( \bZ(t) + \bM_u \vphantom{x^2} \right)^{-1}
		\left[ \bR(\calB) - \bR(\calA) \right]
		\left( \bZ(t) + \bM_u \vphantom{x^2} \right)^{-1}
	\right\|
		\text{.}
\end{equation*}
Next, we once again use the fact that~$\bR(\calB) \succeq \bR(\calA)$ to obtain~$\bZ(t) \succeq \bZ(0) = \bR(\calA)$, effectively removing the dependence on~$t$. Using Cauchy-Schwartz then yields
\begin{equation*}
	\dot{h}(t) \leq \lmax\left[
		\left( \bR(\calA) + \bM_u \right)^{-2} \right]
		\lmax \left[ \bR(\calB) - \bR(\calA) \right]
		\text{,}
\end{equation*}
which can be used in~\eqref{E:epsilonHomotopyE} to get
\begin{equation}\label{E:epsilonBoundImprove}
	\epsilon \leq
	\max_{\substack{\calA \subseteq \calB \subseteq \calO \\
		u \in \calO \setminus \calB}}
	\frac{
		\lmax\left[ \bR(\calB) - \bR(\calA) \right]
	}{
		\lmin\left[ \bR(\calA) + \bM_u \right]^2
	}
		\text{,}
\end{equation}
The inequality in~\eqref{E:epsilonBoundE} is then obtained using~$\norm{\bR(\calB) - \bR(\calA)} \leq \lmax(\sum_{u \in \calO} \bM_u)$ and~$\lmin\left[ \bR(\calA) + \bM_u \right] \geq \lmin\left[ \bM_\emptyset \right]$.
\end{proof}

From Theorem~\ref{T:epsilonBound} and Lemma~\ref{T:preserveSM}, we obtain:

\begin{corollary}\label{T:epsilonBoundCombination}

The objective of~\eqref{P:eig} is $\epsilon$-supermodular for
\begin{equation}\label{E:epsilonEig}
	\epsilon \leq \sum_{k = 0}^{N-1} \theta_k
	\frac{
		\lmax(\sum_{u \in \calO} \bM_{u,m+k})
	}{
		\lmin\left[ \bM_{\emptyset,m+k} \right]^{2}
	}
		\text{,}
\end{equation}

\end{corollary}

In the sequel, we particularize these results for the estimation problems~(i) and~(ii) in Section~\ref{S:problem}, showing that we we obtain supermodular-like guarantees when a certain measure of signal-to-noise ratio~(SNR) is small. We also argue that this is in fact the typical setting in which KFs are useful in practice.

\section{NEAR-OPTIMAL SENSOR SELECTION}
\label{S:nearOptimalSS}

\begin{algorithm}[tb]
	\caption{Cardinality~$r$ greedy set selection for~\eqref{P:sensorSelection}}
	\label{A:GreedyAlgorithm}
	\setlength{\baselineskip}{1.2\baselineskip}
\begin{algorithmic}
	\State Initialize~$\calG_0 \leftarrow \emptyset$

	\For{$j = 1,\dots,r$}
		\State $\displaystyle
			u \leftarrow \argmin_{w \in \calO \setminus \calG_j}
			\sum_{k = 0}^{N-1} \theta_k\, h\Big[ \bY_{m+k}(\calG_j \cup \{w\}) \Big]$
				
		\vspace*{0.5\baselineskip}
		
		\State $\calG_{j+1} \leftarrow \calG_j \cup \{u\}$
	\EndFor
\end{algorithmic}
\end{algorithm}

In this section, we obtain approximation guarantees for the greedy optimization of the MSE and worst-case error in filtering and smoothing problems. More specifically, we provide near-optimal certificates for the solution~$\calG_r$ obtained using Algorithm~\ref{A:GreedyAlgorithm} with~$\bY$ taken from Proposition~\ref{T:filtering}~(filtering) or Proposition~\ref{T:smoothing}~(smoothing) and with the trace or spectral norm scalarizations~$h$~(as in Section~\ref{S:KFSS}).

\subsection{Greedy sensor selection for filtering is near-optimal}
\label{S:greedyFiltering}

The main result of this section stems from applying Proposition~\ref{T:filtering} to Corollaries~\ref{T:alphaBoundCombination} and~\ref{T:epsilonBoundCombination}.

\begin{theorem}\label{T:nearOptimalFiltering}

Let~$\bY_k(\calA) = \bP_k(\calA)$, the filtering error covariance from Proposition~\ref{T:filtering}, in~\eqref{P:trace}--\eqref{P:eig} and denote their solutions as~$\widehat{\calT}^\star$ and~$\widehat{\calS}^\star$, respectively. Let~$\widehat{\calT}_g$ and~$\widehat{\calS}_g$ be the~$r$-elements greedy solutions obtained by using Algorithm~\ref{A:GreedyAlgorithm} on~\eqref{P:trace} and~\eqref{P:eig}. Then,
\begin{align*}
	f_T(\widehat{\calT}_g) &\leq (1 - e^{-\hat{\alpha} r/s})
		f_T(\widehat{\calT}^\star)
	\\
	f_S(\widehat{\calS}_g) &\leq (1 - e^{-r/s})
		\left[ f_S(\widehat{\calS}^\star)
		+ s \cdot \hat{\epsilon} \right]
\end{align*}
for~$f_T$ and~$f_S$ the objectives of~\eqref{P:trace} and~\eqref{P:eig} respectively, $\hat{\alpha} \geq \min_{k \leq N-1} \hat{\alpha}_k$, and~$\hat{\epsilon} \leq \sum_{k = 0}^{N-1} \theta_k \hat{\epsilon}_k$ with
\begin{align}
	\hat{\alpha}_k &\triangleq \frac{
		\lmin\left[ \bP_{m+k \vert m+k-1}^{-1} \right]
	}{
		\lmax\left[ \bP_{m+k \vert m+k-1}^{-1}
			+ \sum_{u \in \calO} \bV_u \right]
	}
		\label{E:alphaFiltering}
	\\
	\hat{\epsilon}_k &\triangleq \lmax\left(
		\sum_{u \in \calO} \bV_u \right)
		\lmax\left( \bP_{m+k \vert m+k-1} \right)^{2}
		\label{E:epsilonFiltering}
\end{align}
where~$\bV_u = \bH_u^T \bR_{v,u}^{-1} \bH_u$.

\end{theorem}

\begin{figure}[tb]
	\centering
	\includesvg{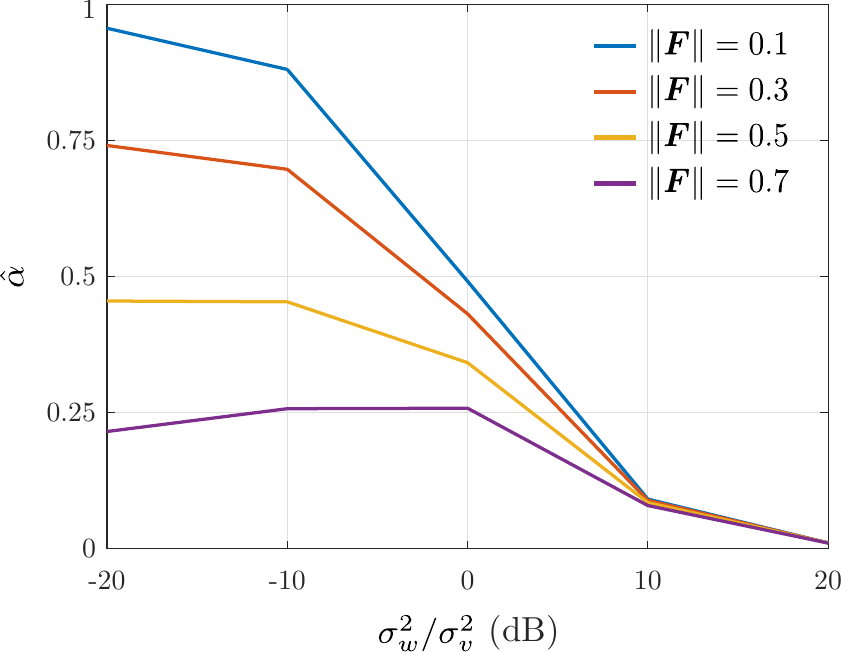}
	\caption{$\alpha$-supermodularity bounds~(Theorem~\ref{T:nearOptimalFiltering}) for different process noise variances and~$\norm{\bF}_2$.}
	\label{F:filterAlpha}
\end{figure}

Theorem~\ref{T:nearOptimalFiltering} shows that greedy sensor selection for state filtering is near-optimal. In fact, it gives explicit bounds on the suboptimality of greedy sensing sets as a function of parameters of the dynamical system. When we do not allow violations of the cardinality constraint, i.e., for~$r = s$, and~\eqref{E:alphaFiltering} is close to one or~\eqref{E:epsilonFiltering} is close to zero, the guarantees from Theorem~\ref{T:nearOptimalFiltering} approach the~$(1-1/e)$-optimality of greedy supermodular minimization.

To see when this occurs specifically, notice that~\eqref{E:alphaFiltering} can be rewritten using the Riccati equation~\eqref{E:Pk} as
\begin{equation*}
	\hat{\alpha}_k = \frac{
		\lmin\left[ \bP_{m+k \vert m+k-1}^{-1} \right]
	}{
		\lmax\left[ \bP_{m+k}^{-1}(\calO) \right]
	}
	= \frac{
		\lmin\left[ \bP_{m+k}(\calO) \right]
	}{
		\lmax\left[ \bP_{m+k \vert m+k-1} \right]
	}
		\text{,}
\end{equation*}
i.e., as the ratio between the \emph{a posteriori} error using all sensors and the \emph{a priori} error. Hence, $\alpha_j$ in~\eqref{E:alphaFiltering} is large when the \emph{a priori} and \emph{a posteriori} error covariance matrices are similar. To get additional insight, consider the particular case in which~$\bPi_0 = \bR_w = \sigma_w^2 \bI$ and~$\bR_v = \sigma_v^2 \bI$. Also, suppose that each state is measured directly, i.e., that~$\calO = \{1,\dots,p\}$ and that~$\bH_u = \be_u^T$, $u \in \calO$, where~$\be_u$ is the $u$-th column of a~$p \times p$ identity matrix. Then, from~\eqref{E:priorPk} and~\eqref{E:Pk}, the matrices of interest take the form
\begin{subequations}
\begin{align}
	\bP_{k \vert k-1} &= \bF \bP_{k-1} \bF^T + \sigma_w^2 \bI
		\label{E:prioriKF2}
	\\
	\bP_{k}(\calO) &=
	\left(
		\bP_{k \vert k-1}^{-1} + \sigma_v^{-2} \bI
	\right)^{-1}
		\label{E:posterioriKF2}
\end{align}
\end{subequations}

Notice that for~\eqref{E:alphaFiltering} to approach one, it must be that~$\bP_{k}(\calO) \approx \bP_{k \vert k-1}$. Since~$\bP_{k \vert k-1}$ depends directly on~$\sigma_w^2$, it is clear from~\eqref{E:posterioriKF2} that this occurs when~$\sigma_v^{2} \gg \sigma_w^{2}$. Additionally, we need~$\bP_{k \vert k-1}$ to be well-conditioned. When~$\sigma_w^2$ is large, this stems directly from~\eqref{E:prioriKF2}. On the other hand, if~$\sigma_w^2 \ll \norm{\bF}$, then~$\bP_{k \vert k-1} \approx \bF \bP_{k-1} \bF^T$ and its condition number depends directly on~$\bF$, since~$\kappa(\bP_{k \vert k-1}) \leq \kappa(\bP_{k-1}) \kappa(\bF)^2$, where~$\kappa$ is the condition number with respect to the spectral norm~\cite{Horn13}. A similar argument applies to~\eqref{E:epsilonFiltering}, which in this simplified setting reduces to
\begin{equation*}
	\hat{\epsilon}_k = \sigma_v^{-2} \lmax\left(
		\bF \bP_{m+k-1} \bF^T + \sigma_w^2 \bI \right)^{2}
		\text{.}
\end{equation*}
Hence, similar conditions as above also guarantee that~$\hat{\epsilon}$ is small. In particular, this occurs when~$\sigma_v^{2}$ is much larger than~$\sigma_w^{2}$ and~$\norm{\bF}$.

\begin{figure}[tb]
	\centering
	\includesvg{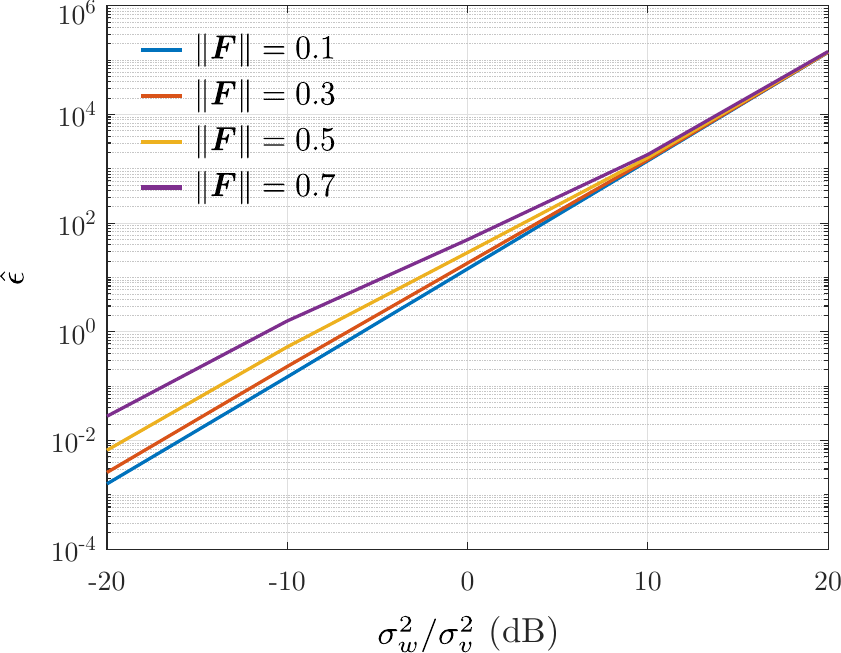}
	\caption{$\epsilon$-supermodularity bounds~(Theorem~\ref{T:nearOptimalFiltering}) for different process noise variances and~$\norm{\bF}_2$.}
	\label{F:filterEpsilon}
\end{figure}

Figures~\ref{F:filterAlpha} and~\ref{F:filterEpsilon} illustrate these observations by showing the values of~$\hat{\alpha}$ and~$\hat{\epsilon}$, averaged over~$100$ dynamical systems realizations, for different values of~$\sigma_w^{2}/\sigma_v^{2}$ and~$\norm{\bF}$~(see Section~\ref{S:Sims} for details). Indeed, the guarantees from Theorem~\ref{T:nearOptimalFiltering} are stronger when~$\sigma_v^{2} \gg \sigma_w^{2}$, $\sigma_v^{2} \gg \norm{\bF}$, and when the system has modes with similar rates~($\kappa(\bF) \approx 1$). It is important to highlight that this is the scenario in which KFs are most useful: if the process noise dominates the estimation error, the system trajectory is mostly random and the choice of sensing subset has little impact on the estimation performance~\cite{Kailath00l, Grewal2014k}. Also notice from Figure~\ref{F:filterAlpha} that for~$\sigma_v^2 > 10\sigma_w^2$ and~$\norm{\bF} = 0.3$, we can obtain supermodular-like guarantees for greedy sensor selection by violating the cardinality constraint by~$35\%$. Since guarantees for problem~\eqref{P:eig} are additive, meaningful performance bounds are obtained on a narrower range of parameters~(Figure~\ref{F:filterEpsilon}).

\begin{remark}

In~\cite{Ye18c}, it was shown that it is impossible to obtain a non-trivial, universal~(system-independent) bound on~$\hat{\alpha}$ unless~$\text{P} = \text{NP}$. To do so, they provide a reduction from the problem of \emph{exact cover by 3-sets}~(X3C) that relies on a dynamical system with noiseless outputs. Notice that this is in line with the results from Theorem~\ref{T:nearOptimalFiltering} given that~$\hat{\alpha} \to 0$ as~$\sigma_v^2 \to 0$, i.e., the guarantees based on approximate supermodularity become vacuous for noise-free outputs. When this is not the case, however, Theorem~\ref{T:nearOptimalFiltering} provides non-trivial near-optimal certificates for greedy sensor selection as a function of the dynamical system parameters.

\end{remark}

\begin{figure}[tb]
	\centering
	\includesvg{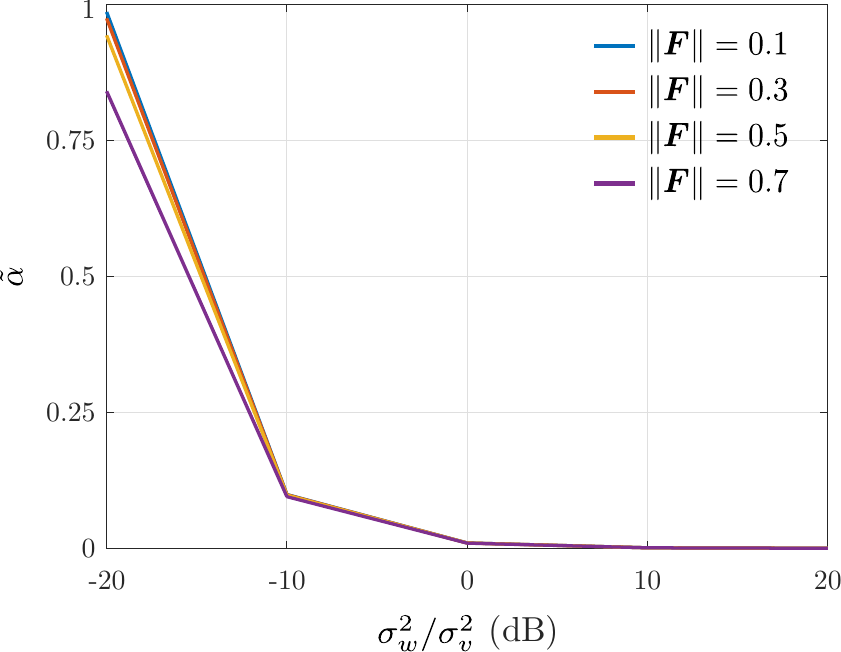}
	\caption{$\alpha$-supermodularity bounds~(Theorem~\ref{T:nearOptimalSmoothing}) for different process noise variances and~$\norm{\bF}_2$.}
	\label{F:smoothAlpha}
\end{figure}

\subsection{Greedy sensor selection for smoothing is near-optimal}
\label{S:greedySmoothing}

We obtain similar results for smoothing by taking the values of~$\bM_{\emptyset,k}$ and~$\bM_{u,k}$ from Proposition~\ref{T:smoothing}.

\begin{theorem}\label{T:nearOptimalSmoothing}

Let~$\bY_k(\calA) = \bQ_k(\calA)$, the smoothing error covariance matrix from Proposition~\ref{T:smoothing}, in~\eqref{P:trace}--\eqref{P:eig} and take~$\widetilde{\calT}^\star$ and~$\widetilde{\calS}^\star$ to be their respective solutions. Let~$\widetilde{\calT}_g$ and~$\widetilde{\calS}_g$ be the~$r$-elements greedy solutions obtained by using Algorithm~\ref{A:GreedyAlgorithm} on~\eqref{P:trace} and~\eqref{P:eig}. Then,
\begin{align}
	f_T(\widetilde{\calT}_g) &\leq (1 - e^{-\tilde{\alpha} r/s})
		f_T(\widetilde{\calT}^\star)
	\\
	f_S(\widetilde{\calS}_g) &\leq (1 - e^{-r/s})
		\left[ f_S(\widetilde{\calS}^\star)
		+ s \cdot \tilde{\epsilon} \right]
		\label{E:smoothingGuarantee}
\end{align}
for~$f_T$ and~$f_S$ the objectives of~\eqref{P:trace} and~\eqref{P:eig} respectively, $\tilde{\alpha} \geq \min_{k \leq N-1} \tilde{\alpha}_k$ and~$\tilde{\epsilon} \leq \sum_{k = 0}^{N-1} \theta_k \tilde{\epsilon}_k$ with
\begin{align}
	\tilde{\alpha}_k &\triangleq
		\frac{
			1/\ell_\textup{max}
		}{
			\lmax\left[ \bC^{-1}
			+ \sum_{u \in \calO} \bPhi_{m+k}^T
			\left( \bI \kron \bV_u \right)
			 \bPhi_{m+k} \right]
		}
		\label{E:alphaSmoothing}
	\\
	\tilde{\epsilon}_k &\triangleq \lmax\left[ \sum_{u \in \calO}
		\bPhi_{m+k}^T \left( \bI \kron \bV_u \right)
		\bPhi_{m+k} \right] \ell_\textup{max}^2
		\label{E:epsilonSmoothing}
\end{align}
where~$\bV_u = \bH_u^T \bR_{v,u}^{-1} \bH_u$, $\bC = \blkdiag(\bPi_0,\bI \kron \bR_w)$, and~$\ell_\textup{max} = \max\left[ \lmax(\bPi_0), \lmax(\bR_w) \right]$.

\end{theorem}

Theorem~\ref{T:nearOptimalSmoothing} shows that greedy sensor selection is also near-optimal for state smoothing. The suboptimality guarantees, however, vary in strength depending on the parameters of the dynamical system. Nevertheless, as with the filtering case from Section~\ref{S:greedyFiltering}, the certificates improve in typical practical settings. For simplicity, let us again analyze the case where~$\bPi_0 = \bR_w = \sigma_w^2 \bI$, $\bR_v = \sigma_v^2 \bI$, and~$\bH_j = \be_j^T$, for~$j \in \calO = \{1,\dots,p\}$. Immediately, $\ell_\text{max} = \sigma_w^2$, $\bC = \sigma_w^2 \bI$, and~\eqref{E:alphaSmoothing} simplifies to
\begin{equation}\label{E:alphaSmoothingSimple}
	\tilde{\alpha}_k = \frac{1}{
		1 + \sigma_w^{2} / \sigma_v^{2}
			\lmax\left(\bPhi_{m+k}^T \bPhi_{m+k} \right)
	}
		\text{.}
\end{equation}

Notice that a noise ratio appears explicitly in the denominator of~\eqref{E:alphaSmoothingSimple}. Hence, as for filtering, it is ready that~$\tilde{\alpha} \approx 1$ when~$\sigma_v^2 \gg \sigma_w^2$. Moreover, $\tilde{\alpha}_j$ depends on the system dynamics through the matrix~$\bPhi$. In fact, this dependence can be made explicit by bounding the norm of~$\bPhi$ in terms of its blocks. Explicitly, $\lmax\left(\bPhi_{m+k}^T \bPhi_{m+k} \right) = \norm{\bPhi_{m+k}} \leq \sqrt{\sum_{j = 0}^{m+k} \norm{\bF}^{2j}}$. Hence, stronger guarantees are obtained when~$\sigma_v^2 \gg \norm{\bF}$. Once again, similar arguments hold for~\eqref{E:epsilonSmoothing}, which reduces to
\begin{equation}\label{E:epsilonSmoothingSimple}
	\tilde{\epsilon}_k = \sigma_w^4 / \sigma_v^{2}
		\lmax\left[ \bPhi_{m+k}^T \bPhi_{m+k} \right]
		\text{.}
\end{equation}

We illustrate the values of~$\tilde{\alpha}$ and~$\tilde{\epsilon}$ for different values of~$\sigma_w^{2}/\sigma_v^{2}$ and~$\norm{\bF}$ in Figures~\ref{F:smoothAlpha} and~\ref{F:smoothEpsilon}~(see Section~\ref{S:Sims} for details). As~\eqref{E:alphaSmoothingSimple} and~\eqref{E:epsilonSmoothingSimple} suggest, Theorem~\ref{T:nearOptimalSmoothing} provides better guarantees when~$\norm{\bF} \sigma_w^{2} / \sigma_v^{2} \ll 1$. In other words, when the measurement noise dominates over both the process noise and the decay rate of the system modes, a scenario of practical value in KF applications~\cite{Kailath00l, Grewal2014k}. Similar to the filtering case, Figure~\ref{F:smoothAlpha} shows that if we violate the cardinality constraints by~$35\%$, we can obtain~$(1-1/e)$-optimality for~$\norm{\bF} = 0.3$ and~$\sigma_v^2 > 75\sigma_w^2$. As in the filtering case, good performance guarantees for problem~\eqref{P:eig} are obtained for a stricter range of parameters~(Figure~\ref{F:smoothEpsilon}).

\begin{figure}[tb]
	\centering
	\includesvg{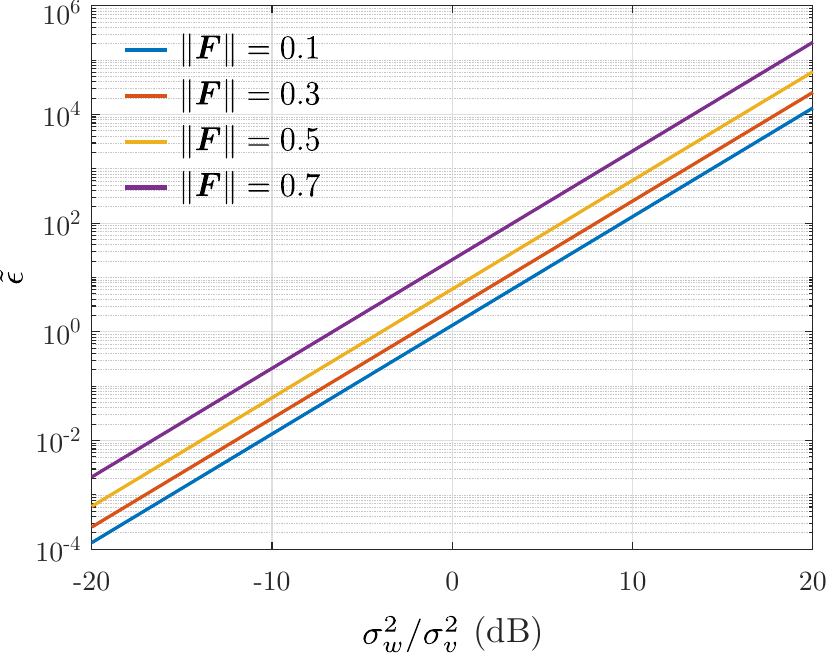}
	\caption{$\epsilon$-supermodularity bounds~(Theorem~\ref{T:nearOptimalSmoothing}) for different process noise variances and~$\norm{\bF}_2$.}
	\label{F:smoothEpsilon}
\end{figure}

\section{NUMERICAL EXAMPLES}
	\label{S:Sims}

We begin by giving details on the simulations in Figures~\ref{F:filterAlpha}--\ref{F:smoothEpsilon}. We considered dynamical systems with~$n = p = 100$ in which the elements of~$\bF$ were drawn randomly from a standard Gaussian distribution, $\bH_u = \be_u^T$, where~$\be_u \in \setR^p$ is a vector of zeros except for the~$u$-th element, which is one, $\bPi_0 = 10^{-2} \bI$, $\bR_w = \sigma_w^2 \bI$, and~$\bR_v = \bI$. The state transition matrix~$\bF$ was normalized so that its spectral norm is in~$[0.1,0.9]$ and the process noise~$\sigma_w^2$ varied in~$[0.01, 100]$. For Figures~\ref{F:filterAlpha} and~\ref{F:filterEpsilon}~(filtering), we minimize the $10$-steps average error~[$N = 10$ and~$\theta_k = 1$ in~\eqref{P:sensorSelection}]. For Figures~\ref{F:smoothAlpha} and~\ref{F:smoothEpsilon}~(smoothing), we minimize the error at the end of the window~[$N = 10$, $\theta_{9} = 1$, and~$\theta_k = 0$ for~$k < 9$ in~\eqref{P:sensorSelection}].

\begin{figure}[tb]
	\centering
	\includesvg{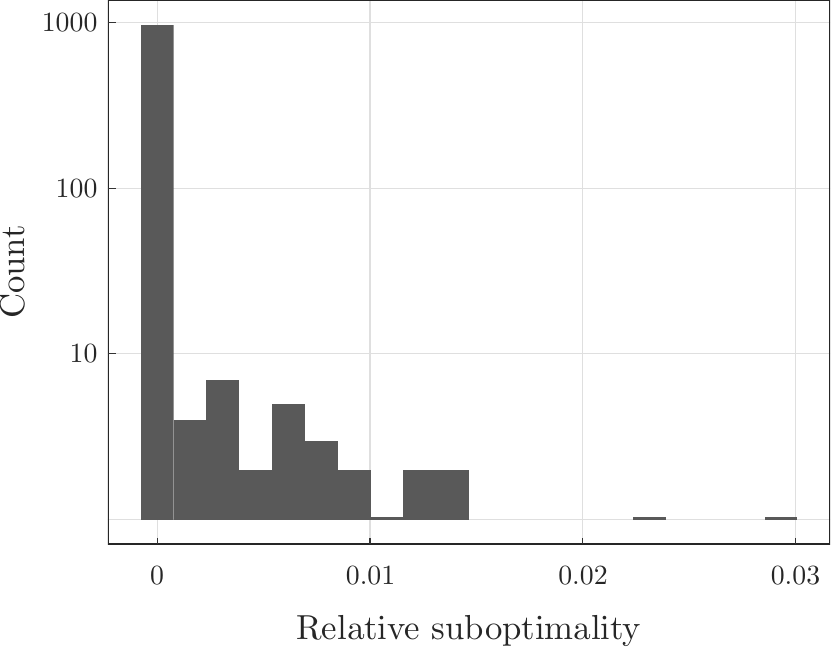}
	\caption{Relative suboptimality of greedy sensor selection using the MSE~[\eqref{P:trace}] for filtering~($1000$ system realizations).}
	\label{F:greedyMSEFilter}
\end{figure}

\begin{figure}[tb]
	\centering
	\includesvg{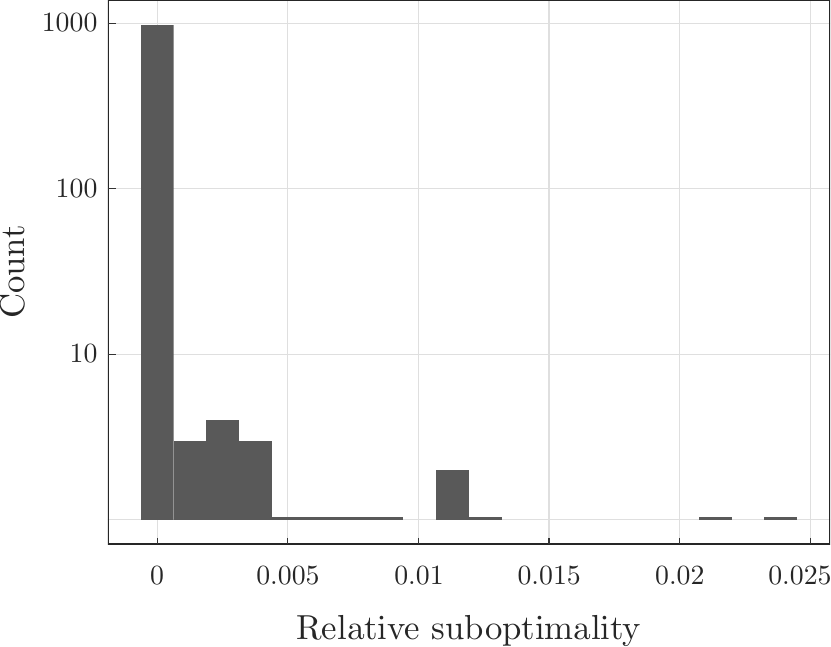}
	\caption{Relative suboptimality of greedy sensor selection using the MSE~[\eqref{P:trace}] for smoothing~($1000$ system realizations).}
	\label{F:greedyMSESmooth}
\end{figure}

Figures~\ref{F:filterAlpha}--\ref{F:smoothEpsilon}, along with Theorems~\ref{T:nearOptimalFiltering} and~\ref{T:nearOptimalSmoothing}, illustrate the wide range of parameters over which good performance certificates can be provided for greedy sensor selection. It is worth noting, however, that these are worst-case guarantees and that better results are common in practice. To illustrate this point, we evaluate the relative suboptimality of greedily selected sensing sets for both filtering and smoothing over~$1000$ system realizations. Explicitly, we evaluate
\begin{equation*}
	\nu^\star(\calG) = \frac{f(\calX^\star) - f(\calG)}{f(\calX^\star)}
		\text{,}
\end{equation*}
where~$\calG$ and~$\calX^\star$ are the greedy and optimal solutions of~\eqref{P:sensorSelection}, respectively. In all experiments, we take~$r = s$, i.e., we do not consider violations of the cardinality constraint. Since~$\nu^\star$ depends on the optimal sensing set, we restrict ourselves to small dynamical systems~[$n = 10$~states, $p = 10$~outputs, and~$s = 4$ in~\ref{P:sensorSelection}]. However, we now randomly draw the elements of both~$\bF$ and~$\bH$ from standard Gaussian distributions and normalize~$\bF$ so that its norm is~$0.9$. We keep~$\sigma_w^2 = 10^{-2}$ fixed, but draw the measurement noise variances at each output uniformly at random from~$[10^{-2},1]$.

\begin{figure}[tb]
	\centering
	\includesvg{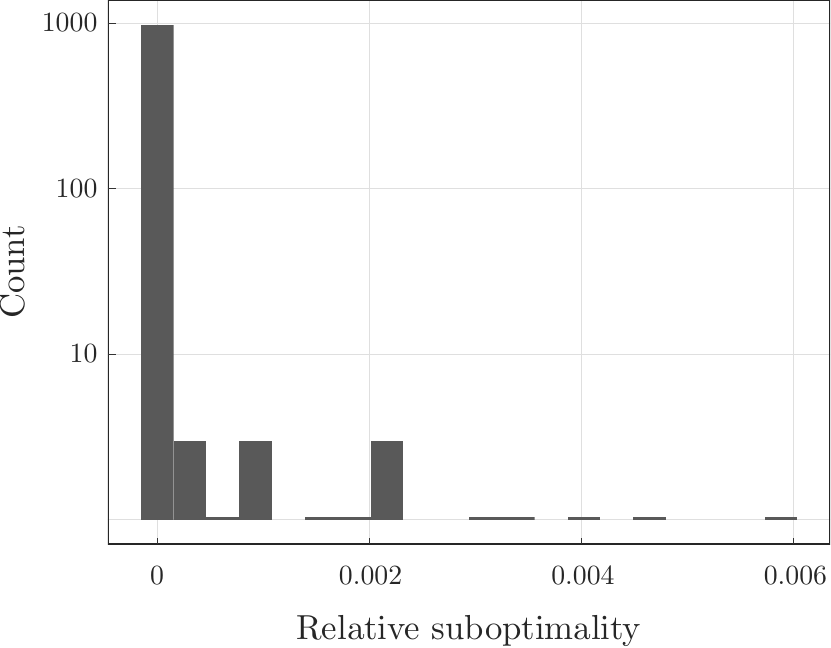}
	\caption{Relative suboptimality of greedy sensor selection using the spectral norm~[\eqref{P:eig}] for filtering~($1000$ system realizations).}
	\label{F:greedyNormFilter}
\end{figure}

\begin{figure}[tb]
	\centering
	\includesvg{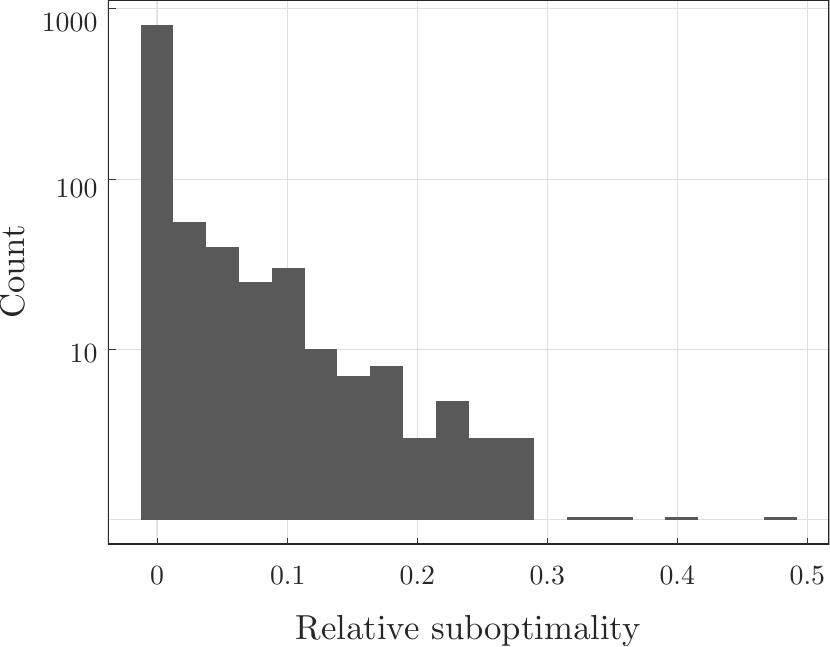}
	\caption{Relative suboptimality of greedy sensor selection using the spectral norm~[\eqref{P:eig}] for smoothing~($1000$ system realizations).}
	\label{F:greedyNormSmooth}
\end{figure}

In Figures~\ref{F:greedyMSEFilter} and~\ref{F:greedyMSESmooth}, we show the value of~$\nu^\star$ for the result of greedily solving problem~\eqref{P:trace} for filtering and smoothing. Notice that the values obtained are considerably lower than the guarantee given in Theorems~\ref{T:nearOptimalFiltering} and~\ref{T:nearOptimalSmoothing}. In fact, they are lower than the~$1-e^{-1} \approx 0.63$ bound for supermodular functions. For both problem, greedy sensor selection found the optimal sampling set in over~$96\%$ of the realizations.

In Figures~\ref{F:greedyNormFilter} and~\ref{F:greedyNormSmooth}, we show results for problem~\eqref{P:eig}. To make sure the system is observable over the time window selected, we use a smaller system~[$n = 5$ states, $p = 10$ outputs, and~$s = 5$ in~\eqref{P:sensorSelection}]. Moreover, since the guarantees for the greedy solution of~\eqref{P:eig} are additive and therefore not as strong as those for~\eqref{P:trace}, we take~$\sigma_w^2 = 10^{-3}$. For the filtering problem, the results are similar to those observed in Figure~\ref{F:greedyMSEFilter} for minimizing the MSE. In over~$98\%$ of the realizations, greedy sensor selection actually obtained the optimal sensing set. On the other hand, the smoothing problem is more challenging for the worst-case error and the results obtained are poorer. This is especially due to the fact that~$f(\widetilde{\calS}^\star) \ll s \cdot \tilde{\epsilon}$ in~\eqref{E:smoothingGuarantee} for this example, so that the guarantees we obtain are less strict. Still, greedy sensor selection found the optimal sensing set in~$77\%$ of the realizations.

\begin{figure}[tb]
	\centering
	\includesvg{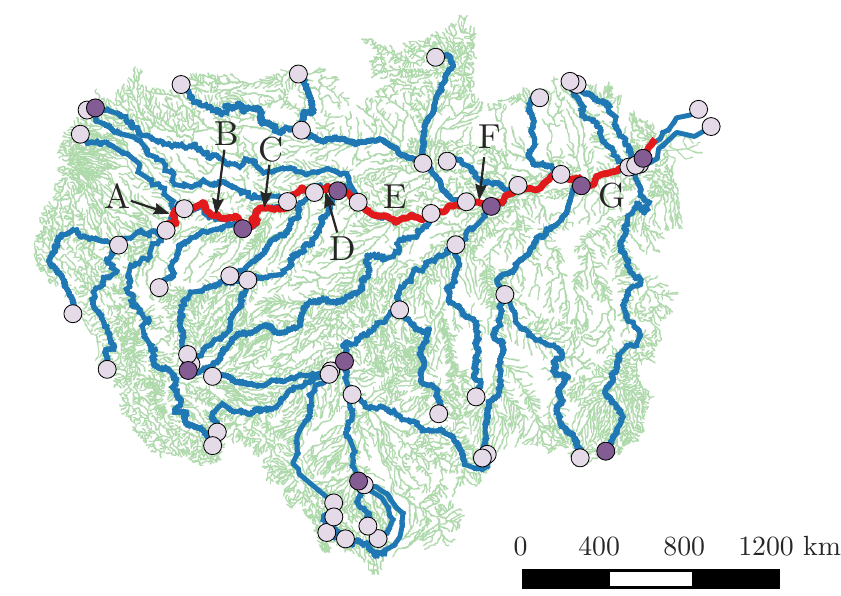}
	\caption{Amazon basin, sensor pool~(light circles), and greedy sensing set~(dark circles).}
	\label{F:amazon}
\end{figure}

\subsection{Application: sensing for the Amazon basin}

We conclude these numerical examples by illustrating the use of greedy sensor selection for water monitoring in the Amazon basin. The Amazon drainage basin, showed in green in Figure~\ref{F:amazon}, covers $7.5$~million~$\text{km}^2$ and is composed of over~$7000$ tributaries. In our experiments, we use a simplified network of the basin~(blue curve in Figure~\ref{F:amazon}) obtained by smoothing the original map~\cite{Muller99u}.

Using this river network, we construct a weighted directed tree~$\calG = (\calV,\calE)$. The set of vertices~$\calV$ is composed of the possible sensor positions shown in Figure~\ref{F:amazon} and an additional node midway between each sensor~($\abs{\calV} = 127$). The set of edges~$\calE$ is composed of ordered pairs from~$\calV \times \calV$ such that~$(i,j) \in \calE$ if water flows from node~$i$ to node~$j$. We assume all river flow toward the ocean. The adjacency matrix of~$\calG$ is then given by
\begin{equation}\label{E:amazonAdjacency}
	[\bA]_{ij} =
	\begin{cases}
		\exp\left( \norm{\bz_i - \bz_j}^2 / \sigma^2 \right)
			\text{,} &\text{if } (i,j) \in \calE
		\\
		0
			\text{,} &\text{otherwise}
	\end{cases}
\end{equation}
where~$\bz_i \in \setR^2$ is the position of node~$i$ on the map and~$\sigma^2 = 10$ is a smoothness parameter. We also define its Laplacian as~$\bL = \bD - \bA$, where~$\bD$ is a diagonal matrix whose elements are the sum of the columns of~$\bA$, and its symmetrized Laplacian as~$\bL^\prime = \bD^\prime - (\bA + \bA^T)$, where~$\bD^\prime$ is a diagonal matrix whose elements are the sum of the columns of~$\bA + \bA^T$.

\begin{figure}[tb]
	\centering
	\includesvg{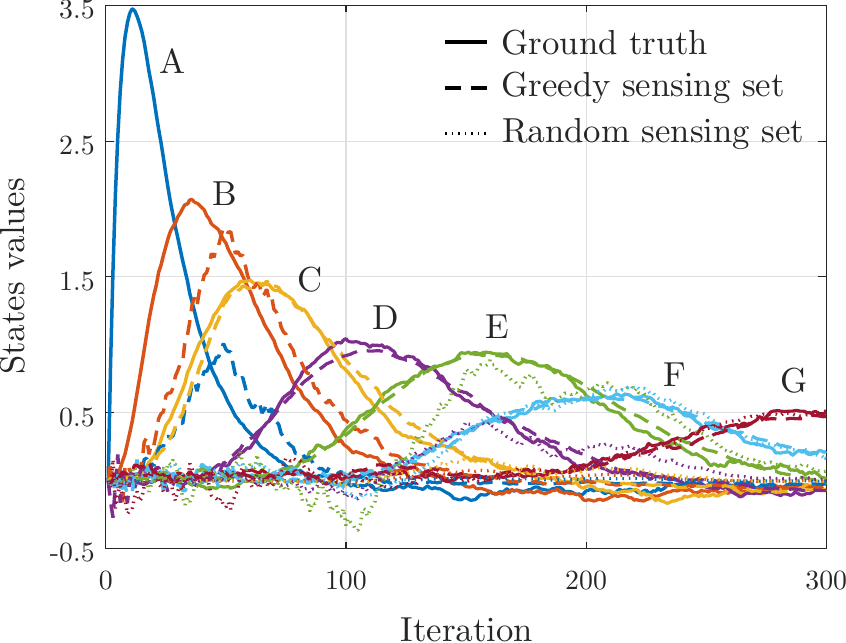}
	\caption{Estimated unobserved states along the Amazon river. States positions are indicated in Figure~\ref{F:amazon}.}
	\label{F:application}
\end{figure}

Using~$\bA$ from~\eqref{E:amazonAdjacency} and the Laplacians defined above, we construct a dynamical system whose states~$\bx_k$ evolve according to~\eqref{E:dynamics} with
\begin{equation}\label{E:amazonDynamics}
	\bF = 0.9 \exp(-\bL \Delta t) + 0.099 \exp(- \bL^\prime \Delta t)
		\text{,}
\end{equation}
where~$\Delta t = 0.1$ is the sampling period. The dynamics in~\eqref{E:amazonDynamics} are a combination of two processes: the first term corresponds to the advection process by which water flows to the ocean and the second term corresponds to a diffusion process. The combination coefficients are chosen so that the system is marginally stable~($\norm{\bF} < 1.02$). The process noise~$\bw_k$ is drawn from a zero-mean Gaussian random variable with covariances~$\bR_w = \sigma_w^2 \bI$, for~$\sigma_w^2 = 10^{-4}$. The true initial state of the system is a spike at the beginning of the Amazon river~(the red curve in Figure~\ref{F:amazon}), i.e., $\bx_0$ is a vector of zeros except at position~$51$, where it is~$10$. For the KF, however, we assume the initial state is a zero-mean Gaussian random variable with covariance~$\bPi_0 = \bI$.

We greedily select~$s = 10$ sensors~(dark circles in Figure~\ref{F:amazon}) from a pool of~$p = 64$ sensors whose positions are shown in Figure~\ref{F:amazon}. The sensors observe a noisy version of their respective states, i.e., $y_{u,k} = e_u^T \bx_k + v_{u,k}$ for~$u \in \calO$, where~$\be_u$ is a vector of zeros except for the $u$-th element that is one and $v_{u,k}$ is a zero-mean Gaussian random variable independent of~$v_{w,\ell}$ for all~$w \neq u$ or~$\ell \neq k$ with variance~$\sigma_v^2 = 10^{-1}$.

The state trajectories estimated using Kalman filtering are shown in Figure~\ref{F:application}. We show only the values at~$7$ unobserved portions along the Amazon river~(indicated in Figure~\ref{F:amazon}). We compare the performance of a greedy sensing set selected to optimize the average~$200$-steps MSE and a sensing set containing one randomly selected sensor for each of the first ten levels of the Amazon basin tree. Notice from Figure~\ref{F:application} that by portion~$C$, the results using the greedy sensing set track closely the true states of the system, whereas the random sensing only detects the change in the water on the second half of the river~(position~$E$). In terms of error, the average~$200$-steps MSE for a full sensing solution~($\calX = \calO$) is~$185.67$. The greedy sensing set, which contains only~$15\%$ of the possible sensors, achieves a cost of~$1005.03$, an almost~$50\%$ reduction over the~$1978.22$ obtained by the random set. Since the modes of this system have similar decays~(the condition number of~$\bF$ is~$1.3$), the sensing set obtain by greedily solving~\eqref{P:logdet} differs in a single sensor from that obtained using~\eqref{P:trace}. Thus, they achieve similar MSE performances. As discussed in Remark~\ref{R:logdet}, this is not necessarily the case in general.

\section{CONCLUSION}

This work studied the filtering/smoothing sensor selection problems for both the MSE and the worst-case error and provided near-optimal guarantees to their greedy solutions. To do so, it relied on two concepts of approximate supermodularity, $\alpha$- and $\epsilon$-supermodularity, and derived approximation bounds for the greedy minimization of these functions. By bounding~$\alpha$ and~$\epsilon$ for the MSE and the spectral norm of the error covariance matrix, it showed that filtering/smoothing sensor selection problem can be solved near-optimally. In fact, in typical application scenarios, the guarantees approach the~$1-1/e$ certificate of supermodular functions. These results justify the use of greedy sensor selection without the need for surrogate cost functions, such as the~$\log\det$. We expect that other cost functions have similar approximate supermodular behavior and that this theory can be used to give approximation certificates for other control problems. Of particular interest is a formulation that takes the controller inputs into account when designing the sensing set. The underlying theory developed here could also potentially be applied to problems involving joint state/input estimation and/or input and state constraints.

\bibliographystyle{IEEEbib}
\bibliography{IEEEabrv,gsp,sp,math,control,stat}

\appendix

\section*{Proof of Proposition~\ref{T:filtering}}

\begin{proof}

The filtered estimate~$\bxh_k(\calX)$ can be obtained from a previous estimate~$\bxh_{k-1}$ using the KF iteration
\begin{equation}\label{E:measKF}
	\bxh_{k}(\calX) = \bF \bxh_{k-1} + \sum_{u \in \calX}
		\bK_{u,k} \left[ \by_{u,k} - \bH_u \bF \bxh_{k-1} \right]
		\text{,}
\end{equation}
where~$\bK_{u,k} = \bP_{k \vert k-1} \bH_u^T \bE_{u,k}^{-1}$ is the Kalman gain of measurement~$u \in \calO$, with~$\bE_{u,k} = \bH_u \bP_{k \vert k-1} \bH_u^T + \bR_{v,u}$ denoting the innovation covariance matrix and
\begin{equation}\label{E:priorPk}
	\bP_{k \vert k-1} = \bF \bP_{k-1} \bF^T + \bR_w
\end{equation}
denoting the \emph{a priori} error covariance matrix. Then, $\bP_k(\calX)$ is obtained directly from the KF recursion as
\begin{equation}\label{E:Pk}
	\bP_{k}(\calX) = \left[ \bP_{k \vert k-1}^{-1} +
		\sum_{u \in \calX} \bH_u^T \bR_{v,u}^{-1} \bH_u \right]^{-1}
		\text{,}
\end{equation}
where~$\bP_{k-1} = \E \left[ \bxh_{k-1} - \bx_{k-1} \right] \left[ \bxh_{k-1} - \bx_{k-1} \right]^T$ is the error covariance matrix of the state estimate and~$\bP_{0 \vert -1} = \bPi_0$~\cite{Kailath00l}.
\end{proof}

\section*{Proof of Proposition~\ref{T:smoothing}}

\begin{proof}

Instead of relying on a Kalman smoother, we obtain the error covariance matrix by analyzing this problem in the batch setting. Note that both methods solve the same problem and therefore lead to the same solution. Proceeding as in~\cite{Tzoumas16s}, note from~\eqref{E:dynamics} that estimating~$\bxb_k$, i.e., all states~$\bx_j$ for~$j \leq k$, is equivalent to estimating the first state~$\bx_{0}$ and the process noises~$\left\{ \bw_j \right\}_{j \leq k-1}$. Indeed, we can recover the smoothed states using~$\bxt_k = \bPhi_k \bzb_k$, for~$\bzb_k = \vect{cccc}{\bx_{0}^T & \bw_{0}^T & \cdots & \bw_{k-1}^T}^T$ and~$\bPhi_k$ as in~\eqref{E:phik}.

The estimation of~$\bzb_k$ can be cast as a stochastic estimation problem by lifting. Formally, using the stacked~$\byb_{u,k} = \vect{ccc}{\by_{u,0}^T & \cdots & \by_{u,k}^T}^T$
and~$\bvb_{u,k} = \vect{ccc}{\bv_{u,0}^T & \cdots & \bv_{u,k}^T}^T$, it is straightforward that
\begin{equation}\label{E:yGlobal}
	\vect{c}{\byb_{u_1,k} \\ \vdots \\ \byb_{u_s,k}} =
		\underbrace{\vect{c}{\bI \kron \bH_{u_1} \\ \vdots \\ \bI \kron \bH_{u_s}} \bPhi_k}_{\bO_k(\calX)}
		\bzb_k
			+ \vect{c}{\bvb_{u_1,k} \\ \vdots \\ \bvb_{u_s,k}}
		\text{,}
\end{equation}
where~$\calX = \{u_1, \dots, u_s\}$ is the observed output subset and~$\kron$ denotes the Kronecker product. The minimum MSE incurred from estimating~$\bzb$ from~$\byb$ is then given by the trace of
\begin{equation}\label{E:placeK}
	\bK(\calX) = \left[ \bC^{-1} +
		\bO_k(\calX)^T \blkdiag\left( \bI \kron \bR_{v,u} \right)^{-1}
			\bO_k(\calX) \right]^{-1}
		\text{,}
\end{equation}
for~$u \in \calX$, $\bO_k$ in~\eqref{E:yGlobal}, $\bC = \blkdiag(\bPi_{0},\bI \kron \bR_w)$, and $\blkdiag(\bX,\bY)$ denoting the block diagonal matrix whose diagonal blocks are~$\bX$ and~$\bY$~\cite{Kailath00l}.

To obtain the form in~\eqref{E:errorCovariance}, use the fact that~$\blkdiag(\bA_i)^{-1} = \blkdiag(\bA_i^{-1})$ and the mixed product property of the Kronecker product~\cite{Horn13} to get
\begin{multline}\label{E:preMSE}
	\bO_k(\calX)^T \blkdiag\left( \bI \kron \bR_{v,u} \right)^{-1} \bO_k(\calX) ={}
	\\
	\sum_{u \in \calX}
		\bPhi_k^T \left( \bI \kron \bH_{u}^T \bR_{v,u}^{-1} \bH_{u} \right) \bPhi_k
		\text{.}
\end{multline}
Substituting~\eqref{E:preMSE} in~\eqref{E:placeK} yields~\eqref{E:smoothingM}.
\end{proof}

\section*{Proof of Theorem~\ref{T:alphaGreedy}}

\begin{proof}
Since~$f$ is monotone decreasing, it holds for every set~$\calG_j$ that
\begin{align}
	f(\calX^\star) &\geq f(\calX^\star \cup \calG_j)
	\notag\\
	{}&= f(\calG_j) -
		\sum_{i = 1}^{s} \left[ f(\calT_{i-1}) -
			f(\calT_{i-1} \cup v_i^\star) \right]
		\label{E:greedyTelescopic}
		\text{,}
\end{align}
where~$\calT_{0} = \calG_j$, $\calT_{i} = \calG_j \cup \{ v_{1}^\star, \dots, v_{i}^\star \}$, $i = 1,\dots,s$, and $v_i^\star$ is the~$i$-th element of~$\calX^\star$. Notice that this holds regardless of the order in which the~$v_i^\star$ are taken. Since~$f$ is~$\alpha$-supermodular and~$\calG_j \subseteq \calT_i$ for all $i$, the incremental gains in~\eqref{E:greedyTelescopic} can be bounded using~\eqref{E:supermodularity} to get
\begin{equation}\label{E:preGreedyRecursion}
	f(\calX^\star) \geq f(\calG_j) -
		\alpha^{-1} \sum_{i = 1}^{s} \left[ f(\calG_j) -
			f(\calG_j \cup v_i^\star) \right]
		\text{.}
\end{equation}
To proceed, we use the fact that~$\calG_{j+1} = \calG_j \cup \{u\}$ is constructed by the greedy procedure in~\eqref{E:greedy} so as to maximize the incremental gains in~\eqref{E:preGreedyRecursion}. It therefore holds that
\begin{equation}\label{E:greedyRecursion1}
	f(\calX^\star) \geq f(\calG_j) -
		\alpha^{-1} s \left[ f(\calG_j) - f(\calG_{j+1}) \right]
		\text{.}
\end{equation}
By adding and subtracting~$f(\calX^\star)$ in the brackets of~\eqref{E:greedyRecursion1}, \eqref{E:greedyRecursion1} yields a recursion for the distance to the optimal value~$\delta_j = f(\calX^\star) - f(\calG_j)$ given by
\begin{equation*}
	\delta_j \geq
		\alpha^{-1} s \left[ \delta_{j} - \delta_{j+1} \right]
	\Rightarrow
	\delta_{j+1} \geq \left( 1 - \frac{1}{\alpha^{-1} s} \right)
		\delta_{j}
		\text{.}
\end{equation*}
Note that since~$f$ is non-positive, $\delta_j \leq 0$ for all~$j$ by the optimality of~$\calX^\star$.

The expression in~\eqref{E:alphaRelOpt} yields directly from this recursion by noticing that since~$f$ is normalized, we have that~$\delta_0 = f(\calX^\star) - f(\emptyset) = f(\calX^\star)$. Immediately, since Algorithm~\ref{A:GreedyAlgorithm} is used for~$r$ iterations, we obtain
\begin{equation}\label{E:greedyFinalRecursion}
	f(\calX^\star) - f(\calG_r) \geq
	\left( 1 - \frac{\alpha}{s} \right)^{r}
		f(\calX^\star)
		\text{.}
\end{equation}
Using the fact that~$1 - x \leq e^{-x}$, the expression in~\eqref{E:greedyFinalRecursion} can be rearranged into~\eqref{E:alphaRelOpt}.
\end{proof}

\section*{Proof of Theorem~\ref{T:epsilonGreedy}}

\begin{proof}

Using the assumption that~$f$ is monotone decreasing, we write
\begin{align}\label{E:greedyTelescopic2}
	f(\calX^\star) &\geq f(\calX^\star \cup \calG_j)
	\notag\\
	&= f(\calG_j) - \sum_{i = 1}^{s} \left[
		f(\calT_{i-1}) - f(\calT_{i-1} \cup \{v_i^\star\})
	\right]
		\text{,}
\end{align}
where again~$\calT_{0} = \calG_j$, $\calT_{i} = \calG_j \cup \{ v_{1}^\star, \dots, v_{i}^\star \}$, for~$i = 1,\dots,s$, and~$v_i^\star$ is the $i$-th element of~$\calX^\star$. Observe that the order in which the~$v_i^\star$ occur in the telescopic sum is irrelevant. Since~$f$ is~$\epsilon$-supermodular and~$\calG_j \subseteq \calT_i$ for all~$i$, \eqref{E:epsilonSM} can be used to bound the incremental gains in~\eqref{E:greedyTelescopic2}. Explicitly,
\begin{equation*}
	f(\calX^\star) \geq f(\calG_j) -
		\sum_{i = 1}^{s} \left[ f(\calG_j) - f(\calG_j \cup \{v_i^\star\}) + \epsilon \right]
		\text{.}
\end{equation*}
Since~$\calG_{j+1} = \calG_j \cup \{u\}$ is constructed in~\eqref{E:greedy} so as to minimize~$f(\calG_{j+1})$, it holds that
\begin{equation}\label{E:greedyRecursion12}
	f(\calX^\star) \geq f(\calG_j) -
	s \left[ f(\calG_j) - f(\calG_{j+1}) + \epsilon \right]
		\text{.}
\end{equation}

Finally, a recursion for the optimality gap~$\delta_j^\prime = f(\calG_j) - f(\calX^\star)$ is obtained from~\eqref{E:greedyRecursion12} by adding and subtracting~$f(\calX^\star)$ in the brakets. Explicitly,
\begin{equation}\label{E:greedyRecursion7}
	\delta_j^\prime \leq s\left( \delta_{j}^\prime - \delta_{j+1}^\prime
		+ \epsilon \right)
	\Rightarrow
	\delta_{j+1}^\prime \leq \left( 1 - \frac{1}{s} \right) \delta_{j}^\prime
		+ \epsilon
		\text{.}
\end{equation}
Notice that since~$f$ is non-positive, $\delta_j^\prime \geq 0$ for all~$j$ due to the optimality of~$\calX^\star$. The solution of~\eqref{E:greedyRecursion7} after~$r$ steps is
\begin{equation}\label{E:greedyRecursionSolve}
	\delta_{r}^\prime \leq
	\left( 1 - \frac{1}{s} \right)^{r} \delta_{0}^\prime
	+ \epsilon \sum_{j = 0}^{r-1} \left( 1 - \frac{1}{s} \right)^j
\end{equation}
Evaluating the geometric series in~\eqref{E:greedyRecursionSolve} yields
\begin{equation*}
	\delta_{r}^\prime \leq
	\left( 1 - \frac{1}{s} \right)^{r} \delta_{0}^\prime
	+ s \left[ 1 - \left( 1 - \frac{1}{s} \right)^{r} \right]
		\epsilon
\end{equation*}
and under the assumption that~$f$ is normalized, i.e., for~$\delta_0^\prime = f(\emptyset) - f(\calX^\star) = -f(\calX^\star) \geq 0$, we obtain
\begin{equation*}
	f(\calG_r) \leq
		\left[ 1 - \left( 1 - \frac{1}{s} \right)^{r} \right]
		\left[ f(\calX^\star) + s \cdot \epsilon \right]
		\text{.}
\end{equation*}
Using once again the fact that~$1 - x \leq e^{-x}$ yields~\eqref{E:epsilonRelOpt}.
\end{proof}

\section*{Proof of Lemma~\ref{T:preserveSM}}

\begin{proof}
We proceed with the proof case-by-case:
\begin{enumerate}[(i)]

\item From the definition of $\alpha$-supermodularity in~\eqref{E:alphaSM}, for $\calA \subseteq \calB \subseteq \calO$ and~$u \in \calO \setminus \calB$ it holds that
\begin{align*}
	g\left( \calA \right) - g\left( \calA \cup \{u\} \right) &=
		\sum_i \theta_i \left[ f_i\left( \calA \right)
		- f_i\left( \calA \cup \{u\} \right) \right]
	\\
	{}&\geq
		\sum_i \alpha_i \theta_i
			\left[ f_i\left( \calB \right) - f_i\left( \calB \cup \{u\} \right) \right]
		\text{,}
\end{align*}
where the constant factors of the affine transformations cancel out. Since~$\alpha_i \geq \min(\alpha_i)$, we obtain
\begin{multline*}
	g\left( \calA \right) - g\left( \calA \cup \{u\} \right) \geq{}
	\\
	\min(\alpha_i) \sum_i \theta_i
		\left[ f_i\left( \calB \right) - f_i\left( \calB \cup \{u\} \right) \right]
		\geq
	\\
	\min(\alpha_i) \left[ g\left( \calB \right) - g\left( \calB \cup \{u\} \right) \right]
		\text{.}
\end{multline*}

\item Similarly, using the definition of $\epsilon$-supermodularity in~\eqref{E:epsilonSM}, for $\calA \subseteq \calB \subseteq \calV$ and~$u \in \calO \setminus \calB$ we obtain
\begin{multline*}
	g\left( \calA \right) - g\left( \calA \cup \{u\} \right) \geq{}
	\\
	\sum_i \theta_i \left[ f_i\left( \calB \right)
		- f_i\left( \calB \cup \{u\} \right) \right]
		- \sum_i \theta_i \epsilon_i ={}
	\\
	g\left( \calB \right) - g\left( \calB \cup \{u\} \right)
		- \sum_i \theta_i \epsilon_i
		\text{,}
\end{multline*}
where again the constant factors~$b$ canceled out.\qedhere

\end{enumerate}
\end{proof}

\section*{Proof of Lemma~\ref{T:DeltaLemma}}

\begin{proof}
Start by simplifying~$\Delta_u t(\calX) = t(\calX) - t(\calX \cup \{u\})$ using the matrix inversion lemma. To do so, write
\begin{equation}\label{E:YXu}
	\bY\left( \calX \cup \{u\} \right) = \left[ \bR(\calX) + \bM_u \right]^{-1}
		\text{,}
\end{equation}
letting once again~$\bR(\calX) = \bM_\emptyset + \sum_{u \in \calX} \bM_u$. Since~$\bR(\calX) \succ 0$, but the matrices~$\bM_u$ need not be invertible, we use an alternative form of the matrix inversion lemma~\cite{Henderson81d} to get
\begin{equation}\label{E:YMIL}
	\bY\left( \calX \cup \{u\} \right) =
		\bY(\calX) - \bY(\calX) \bM_u
			\bY\left( \calX \cup \{u\} \right)
		\text{.}
\end{equation}
where we used~\eqref{E:YXu} to write~$\bY(\calX) = \bR(\calX)^{-1}$. Using~\eqref{E:YMIL} and the linearity of the trace~\cite{Horn13}, we can write~$\Delta_u t(\calX)$ as
\begin{equation}\label{E:DeltaMIL}
	\Delta_u t(\calX) =
	\trace \left[
		\bY(\calX) \bM_u \bY\left( \calX \cup \{u\} \right)
	\right]
		\text{.}
\end{equation}

To proceed, let~$\bMt_u = \bM_u + \epsilon \bI \succ 0$ for~$\epsilon > 0$ and define the perturbed version of~\eqref{E:DeltaMIL} as
\begin{equation}\label{E:perturbedDelta}
	\widetilde{\Delta}_{u} t(\calX) = \trace \left[ \bY(\calX) \bMt_u
		\left( \bR(\calX) + \bMt_u \right)^{-1} \right]
		\text{.}
\end{equation}
Notice that~$\widetilde{\Delta}_{u} t \to \Delta_{u} t$ as~$\epsilon \to 0$. Using the invertibility of~$\bMt_u$ and~$\bR(\calX)$, we obtain
\begin{equation*}
	\widetilde{\Delta}_{u} t(\calX) = \trace \left[ \bY(\calX)
		\left( \bY(\calX) + \bMt_u^{-1} \right)^{-1} \bY(\calX) \right]
		\text{.}
\end{equation*}
Since~$\bY(\calX) \succ 0$, its square-root~$\bY(\calX)^{1/2}$ is well-defined and unique~\cite{Horn13}. We can therefore use the circular commutation property of the trace to get
\begin{equation}\label{E:preDeltaBound}
	\widetilde{\Delta}_{u} t(\calX) = \trace \left[ \bY(\calX) \bZ(\calX,u) \right]
		\text{,}
\end{equation}
with~$\bZ(\calX,u) = \bY(\calX)^{1/2} \left[ \bY(\calX) + \bMt_u^{-1} \right]^{-1} \bY(\calX)^{1/2}$. Since both matrices in~\eqref{E:preDeltaBound} are positive definite, we can use the bound from~\cite{Wang92s} to get
\begin{equation*}
	\lmin\left[ \bY(\calX) \right] \trace \left( \bZ \right) \leq
	\widetilde{\Delta}_{u} t(\calX) \leq
	\lmax\left[ \bY(\calX) \right] \trace \left( \bZ \right)
		\text{.}
\end{equation*}
Reversing the manipulations used to obtain~\eqref{E:preDeltaBound} yields
\begin{multline*}
	\lmin\left[ \bY(\calX) \right]
	\trace \left[ \bMt_u \left( \bY(\calX) + \bMt_u \right)^{-1} \right]
	\leq
	\widetilde{\Delta}_{u} t(\calX)
	\leq{}
	\\
	\lmax\left[ \bY(\calX)^{-1} \right]
	\trace \left[ \bMt_u \left( \bY(\calX) + \bMt_u \right)^{-1} \right]
		\text{.}
\end{multline*}
The result in~\eqref{E:DeltaBound} is obtained by continuity as~$\epsilon \to 0$.
\end{proof}

% ========== BIOGRAPHIES ==========
\begin{IEEEbiography}[{\includegraphics[width=1in,height=1.25in]{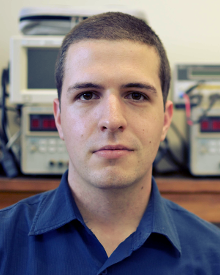}}]{Luiz F. O. Chamon (S'12)}
received the B.Sc. and M.Sc. degree in electrical engineering from the University of S\~{a}o Paulo, S\~{a}o Paulo, Brazil, in 2011 and 2015. In 2009, he was an undergraduate exchange student at the Masters in Acoustics of the \'{E}cole Centrale de Lyon, Lyon, France. He is currently working toward the Ph.D. degree in electrical and systems engineering at the University of Pennsylvania (Penn), Philadelphia. In 2009, he was an Assistant Instructor and Consultant on nondestructive testing at INSACAST Formation Continue. From 2010 to 2014, he worked as a Signal Processing and Statistical Consultant on a project with EMBRAER. His research interest include signal processing, optimization, statistics, and control.
\end{IEEEbiography}

\begin{IEEEbiography}[{\includegraphics[width=1in,height=1.25in]{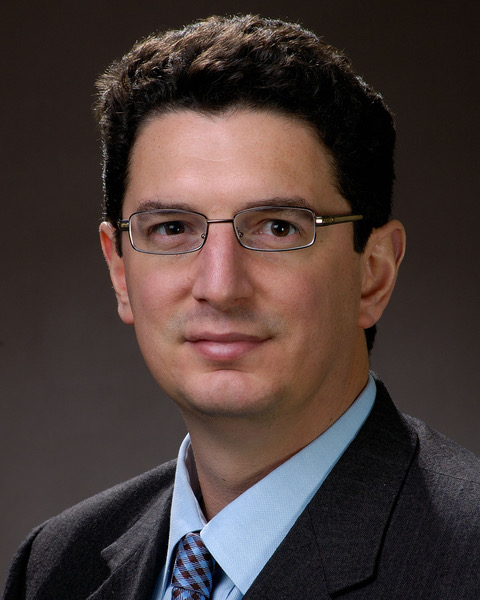}}]{George J. Pappas (F’09)}
received the Ph.D. degree in electrical engineering and computer sciences from the University of California, Berkeley, CA, USA, in  1998. He is currently the Joseph Moore Professor and Chair of the Department of Electrical and Systems Engineering, University of Pennsylvania, Philadelphia, PA, USA. He  also holds a secondary appointment with the Department of Computer and Information Sciences and the Department of Mechanical Engineering and Applied Mechanics. He is a Member of the GRASP Lab and the PRECISE Center. He had previously served as the Deputy Dean for Research with the School of Engineering and Applied Science. His research interests include control theory and, in particular, hybrid systems, embedded systems, cyberphysical systems, and hierarchical and distributed control systems, with applications to unmanned aerial vehicles, distributed robotics, green buildings, and biomolecular networks. Dr. Pappas has received various awards, such as the Antonio Ruberti Young Researcher Prize, the George S. Axelby Award, the Hugo Schuck Best Paper Award, the George H. Heilmeier Award, the National Science Foundation PECASE award and numerous best student papers awards at ACC, CDC, and ICCPS.

\end{IEEEbiography}

\begin{IEEEbiography}[{\includegraphics[width=1in,height=1.25in]{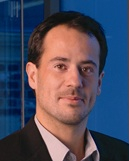}}]{Alejandro Ribeiro}
received the B.Sc. degree in electrical engineering from the Universidad de la Republica Oriental del Uruguay, Montevideo, in 1998 and the M.Sc. and Ph.D. degree in electrical engineering from the Department of Electrical and Computer Engineering, the University of Minnesota, Minneapolis in 2005 and 2007.

From 1998 to 2003, he was a member of the technical staff at Bellsouth Montevideo. After his M.Sc. and Ph.D studies, in 2008 he joined the University of Pennsylvania (Penn), Philadelphia, where he is currently the Rosenbluth Associate Professor at the Department of Electrical and Systems Engineering. His research interests are in the applications of statistical signal processing to the study of networks and networked phenomena. His focus is on structured representations of networked data structures, graph signal processing, network optimization, robot teams, and networked control.

Dr. Ribeiro received the 2014 O. Hugo Schuck best paper award, and paper awards at CDC 2017, 2016 SSP Workshop, 2016 SAM Workshop, 2015 Asilomar SSC Conference, ACC 2013, ICASSP 2006, and ICASSP 2005. His teaching has been recognized with the 2017 Lindback award for distinguished teaching and the 2012 S. Reid Warren, Jr. Award presented by Penn's undergraduate student body for outstanding teaching. Dr. Ribeiro is a Fulbright scholar class of 2003 and a Penn Fellow class of 2015.
\end{IEEEbiography}

\end{document}